\documentclass[11pt,a4paper,reqno]{amsart}

\usepackage[left=22mm,right=22mm,top=22mm,bottom=22mm]{geometry}
\usepackage[nodisplayskipstretch]{setspace}
\usepackage{amsmath,amssymb,amsthm,mathtools,enumitem}
\usepackage{microtype,needspace}
\usepackage[
  colorlinks=true,
  linkcolor=blue,
  citecolor=blue,
  urlcolor=blue
]{hyperref}

\setlist{nosep}
\newcommand{\proofstep}[1]{%
  \par\addvspace{3pt}\Needspace{3\baselineskip}%
  \noindent\textit{#1}\quad
}
\hypersetup{bookmarksnumbered=true,pdfauthor={Tuan Tran}}
\newtheorem{theorem}{Theorem}[section]
\newtheorem*{deterministic}{Deterministic scaling}
\newtheorem{lemma}[theorem]{Lemma}
\newtheorem{proposition}[theorem]{Proposition}
\newtheorem{corollary}[theorem]{Corollary}
\newtheorem{conjecture}{Conjecture}[section]
\theoremstyle{definition}
\theoremstyle{remark}\newtheorem{remark}[theorem]{Remark}
\newcommand{\E}{\mathbb E}
\newcommand{\PP}{\mathbb P}
\newcommand{\R}{\mathbb R}
\newcommand{\one}{\mathbf1}
\newcommand{\ones}{\mathbf{1}}
\newcommand{\cK}{\mathcal K}
\newcommand{\cF}{\mathcal F}
\newcommand{\cE}{\mathcal E}
\newcommand{\cD}{\mathcal D}
\newcommand{\cB}{\mathcal B}

\newcommand{\T}{\mathsf{T}}
\newcommand{\norm}[1]{\lVert#1\rVert}

\DeclareMathOperator{\diag}{diag}
\DeclareMathOperator{\Var}{Var}
\DeclareMathOperator{\Ent}{Ent}

\numberwithin{equation}{section}

\title[Fractional clique decompositions via matrix scaling]{The threshold for fractional clique decompositions of random hypergraphs, via matrix scaling}
\author{Tuan Tran}
\address{School of Mathematical Sciences, University of Science and
Technology of China}
\email{trantuan@ustc.edu.cn}
\date{}

\begin{document}
\begin{abstract}
Let $\delta^*_{k,r}$ be the fractional $K_r^{(k)}$-decomposition
threshold in minimum codegree. For fixed $k\ge2$, $r\ge k+1$ and
$\varepsilon>0$, we prove that every $n$-vertex $k$-uniform
hypergraph $G$ with minimum codegree at least
$(\delta^*_{k,r}+\varepsilon)n$ admits, with high probability, a
fractional $K_r^{(k)}$-decomposition after retaining each edge
independently with probability $p \ge C(\log n/n^{r-k})^{1/(\binom{r}{k}-1)}$.
The bound on $p$ is sharp up to a constant factor, confirming
the conjectured threshold order in the graph case.
The proof develops a probabilistic matrix-scaling approach.
\end{abstract}
\maketitle
\pagestyle{plain}
\flushbottom

\section{Introduction}
\subsection{Background and main result}
Steiner triple systems are among the oldest and best-known objects
in design theory. A \emph{Steiner triple system} on $n$ points is a
collection of triples in which every pair of points belongs to exactly
one triple. In 1847, Kirkman~\cite{Kirkman} proved that such a system
exists exactly when $n\equiv1,3\pmod6$. Viewing the points as vertices
and the pairs as edges identifies a Steiner triple system with a
partition of the edges of $K_n$ into triangles.

More generally, a \emph{$K_r^{(k)}$-decomposition} of a $k$-uniform
hypergraph $G$ is a partition of its edges into
copies of $K_r^{(k)}$, the complete $k$-uniform hypergraph on $r$ vertices.
For every set $S$ of fewer than $k$ vertices, the number of edges
containing $S$ must then be divisible by
$\binom{r-|S|}{k-|S|}$. In the complete host these necessary
conditions are sufficient for all sufficiently large orders:
Wilson~\cite{Wilson} proved the graph case, and
Keevash~\cite{Keevash} proved the general existence theorem for designs.
For incomplete hosts, a central question is which density conditions
make divisibility sufficient.

The fractional relaxation plays a central role in these density
questions. A \emph{fractional
$K_r^{(k)}$-decomposition} of $G$ assigns a nonnegative weight to each
copy of $K_r^{(k)}$ so that, for every edge, the weights of the copies
containing it sum to one. For graphs, Haxell and R\"odl~\cite{HR}
proved that the fractional and integral
packing numbers for a fixed graph differ by $o(n^2)$ in an
$n$-vertex host. The iterative absorption method of Barber, K\"uhn,
Lo and Osthus~\cite{BKLO} converts suitable approximate decompositions
into exact ones. Glock, K\"uhn, Lo, Montgomery and Osthus~\cite{GKLMO}
showed that the asymptotic minimum-degree thresholds for fractional
and integral clique decompositions in graphs agree, with divisibility
required in the integral case.

In the random setting, Yuster~\cite{Yuster-survey}
asked for the threshold for fractional triangle decompositions of
the Erd\H{o}s--R\'enyi random graph $G(n,p)$.
Even the existence of such a threshold is not
guaranteed, since fractional decomposability is not monotone under
edge addition.
Nevertheless, the necessary local condition that every edge lie in
a copy of $K_r$ suggests a natural candidate for the threshold.
This condition is widely believed to be sufficient, motivating the following conjecture.

\begin{conjecture}[cf.~\cite{Yuster-survey,DKP,MS}]
\label{conj:clique-threshold}
For every integer $r\ge3$, there is a constant $C_r>0$ such
that, if $p\ge C_r n^{-2/(r+1)}(\log n)^{1/(\binom{r}{2}-1)}$, then
with high probability $G(n,p)$ has a fractional $K_r$-decomposition.
\end{conjecture}

Earlier results left a polynomial gap from this predicted scale.
Delcourt, Kelly and Postle~\cite{DKP} proved
that $G(n,p)$ has a fractional triangle decomposition with high
probability when $p\ge n^{-1/3+o(1)}$, and a fractional
$K_r$-decomposition for $r\ge4$ when
$p\ge n^{-2/(2r+1)+o(1)}$. Their fractional bounds were limited
by the need to extend copies of $K_r$ to $K_{r+2}$, and they asked
for methods that avoid this restriction. Mahabaduge and
Simkin~\cite{MS} overcame this barrier for triangles, improving the
range to $p\ge n^{-4/11+o(1)}$ by directly correcting triangle
weights with local gadgets.

We resolve the conjecture through a more general result for random
subhypergraphs of dense hosts. Delcourt, Kelly and
Postle~\cite{DKP} proposed this setting for graph clique
packings as a way to connect random and minimum-degree decomposition
problems. The dense-host formulation also connects our result to a
broader program in extremal and probabilistic design theory: in his ICM survey,
Postle~\cite{PostleICM} calls for a unified theory
encompassing minimum-degree conditions, high girth, spread
distributions, and partite settings.

To state our result, let $G_p$ be obtained from a fixed $k$-uniform hypergraph $G$ by retaining each edge independently with probability $p\in[0,1]$. For $G=K_n^{(k)}$, this is the $k$-uniform binomial
random hypergraph $G^{(k)}(n,p)$. The \emph{minimum codegree} $\delta_{k-1}(G)$ is the smallest number of edges containing a $(k-1)$-set of vertices. We write $\delta^*_{k,r}$ for the infimum of all $d\in[0,1]$ such that, for every $\varepsilon>0$, every sufficiently large $n$-vertex $k$-uniform hypergraph $G$ with $\delta_{k-1}(G)\ge(d+\varepsilon)n$ has a fractional $K_r^{(k)}$-decomposition. Estimating $\delta^*_{k,r}$ is notoriously
difficult~\cite{BKLMO,DHLP-counterexamples,DLP,Montgomery}.
The only exact value currently known is $\delta^*_{2,3}=3/4$,
established by Delcourt and Postle~\cite{DP26} in their proof of
Nash-Williams' conjecture.

Our main theorem transfers fractional decomposability from hosts
above the threshold $\delta^*_{k,r}$ to their random subhypergraphs,
at the probability scale suggested by local clique coverage.

\Needspace{7\baselineskip}
\begin{theorem}
\label{thm:minimum-degree}
For all integers $k\ge2$, $r\ge k+1$ and every
$\varepsilon,b>0$, there are constants $C,n_0>0$ such that the following
holds. Let $G$ be a $k$-uniform hypergraph with $n\ge n_0$ vertices and minimum codegree $\delta_{k-1}(G)\ge(\delta^*_{k,r}+\varepsilon)n$.
If $p\ge C(\log n/n^{r-k})^{1/(\binom{r}{k}-1)}$, then with probability at least $1-n^{-b}$, $G_p$ admits a fractional
$K_r^{(k)}$-decomposition.
\end{theorem}

Taking $k=2$ and $G=K_n$ proves
Conjecture~\ref{conj:clique-threshold}.

\Needspace{7\baselineskip}
\medskip
\noindent\textbf{Independent work.}
During the preparation of this manuscript, we became aware of
concurrent work by Joos and Smith~\cite{JS}, who proved that
$G^{(k)}(n,p)$ has a fractional $K_r^{(k)}$-decomposition with high
probability whenever
$p\ge n^{-(r-k)/(\binom{r}{k}-1)+\alpha}$ for any fixed $\alpha>0$.
Their proof modifies and generalizes an iterative
weight-redistribution algorithm proposed by Mahabaduge and
Simkin~\cite{MS}. Our matrix-scaling approach reaches
the logarithmic probability scale and applies to random
subhypergraphs of dense hosts.

\subsection{Matrix scaling}
\label{sec:method}
Our approach corrects random load discrepancies by an exponential change
of weights. We first develop the probabilistic argument for independent columns; this model
explains the size of the correction and provides the proof template for clique decompositions.
The deterministic starting point is the spectral analysis of matrix scaling by Kwok, Lau and Ramachandran~\cite{KLR}.

To see the balancing problem in its simplest form, let $H$ be a bipartite graph with
vertex classes $\{a_1,\ldots,a_N\}$ and $\{b_1,\ldots,b_N\}$,
where every $b_j$ has positive degree. Associate to $b_j$ a
probability vector $v_j\in\R^N$ supported on its neighbors; for
example, give each neighbor weight $1/d_H(b_j)$. Interpreting
$(v_j)_i$ as the weight of $a_ib_j$, the load at every $b_j$ is
already one. These weights form a fractional perfect matching
exactly when $\sum_jv_j=\ones$, where $\ones$ denotes the all-ones
vector. Thus the task is to balance the remaining vertex loads
without changing the supports.

Write $V=(v_1,\ldots,v_N)$. Matrix scaling applies a common
coordinatewise correction and then normalizes each column:
$p_{v_j}(x)_i=(v_j)_i e^{x_i}/(v_j^{\T}e^x)$, where $e^x$ is the
coordinatewise exponential and $\T$ denotes transpose. Every
$p_{v_j}(x)$ is a probability vector with the same support as $v_j$.
The scaled matrix is therefore doubly stochastic (all row and column
sums are one) precisely when $\sum_jp_{v_j}(x)=\ones$. Adding a constant to all coordinates of
$x$ changes none of these columns, so we require $\ones^{\T}x=0$
and call $x$ \emph{centered}. If $\|x\|_\infty\le t$, every
supported weight changes by a factor between $e^{-2t}$ and $e^{2t}$.

Kwok, Lau and
Ramachandran~\cite{KLR} proved the following deterministic statement.

\begin{deterministic}[Kwok--Lau--Ramachandran]
There are constants $c,C>0$ such that the following holds.
Let $V$ be a nonnegative column-stochastic $N\times N$ matrix, where $N\ge 2$, and put $\varepsilon=\|V\ones-\ones\|_\infty$.
Suppose $V$ has second-largest singular value $\sigma_2(V)\le1-\kappa$ for some $0<\kappa\le1$.
If $\varepsilon\log(2N)\le c\kappa^2$, then there is a centered
vector $x\in\R^N$ with
$\|x\|_\infty\le C\varepsilon\log(2N)/\kappa$
such that the matrix with columns $p_{v_j}(x)$ is doubly stochastic.
\end{deterministic}

Suppose now that the columns are independent and balanced in
expectation. Can the realized matrix be balanced by a potential of
the same order as its natural row fluctuations? Our model theorem
shows that a fixed spectral gap and sufficiently small entries suffice.

\Needspace{9\baselineskip}
\begin{theorem}[The model case]\label{thm:T11}
Let $N\ge2$, and let $v_1,\ldots,v_N$ be independent probability
vectors in $\R^N$ with entries at most $\beta$. Suppose that
$\sum_j\E v_j=\ones$ and
\[
 \Big\|\sum_j\E[v_jv_j^{\T}]
             -\frac1N\ones\ones^{\T}\Big\|\le1-\kappa
\]
for some $\kappa>0$. There are constants $c,C>0$, depending only on
$\kappa$, such that the following holds. If $\beta\log(2N)\le c$ and
$\log(2N)\le s\le c/\beta$, then with probability at least $1-e^{-s}$
there is a vector $x\in\R^N$ with $\ones^{\T}x=0$ and
$\|x\|_\infty\le C\sqrt{\beta s}$ such that the $N\times N$ matrix
with entries $B_{ij}=(v_j)_i e^{x_i}/(v_j^{\T}e^x)$ is doubly stochastic.
\end{theorem}

The theorem gives
$\|x\|_\infty=O_\kappa(\sqrt{\beta\log N})$ at polynomially small
failure probability. Scalar Bernstein~\cite{BLM} and a union bound give the
same upper bound for $\|V\ones-\ones\|_\infty$. Even when a realization
satisfies the deterministic spectral and small-imbalance conditions,
substituting this discrepancy estimate into the deterministic bound
above gives an additional factor of $\log N$. We avoid that loss by
estimating the potential itself: coordinate replacement square sums
give concentration, and a resampling comparison controls its mean.
This square-root dependence on the entry scale is the quantitative
feature we seek to preserve for clique decompositions.
We prove Theorem~\ref{thm:T11} in Section~\ref{sec:scaling}.

\subsection{Quantitative inheritance}
\label{sec:seed-intro}
We now transfer this balancing principle to clique decompositions. A fractional clique decomposition is a fractional perfect matching in the auxiliary hypergraph whose vertices are host edges and whose hyperedges are the edge sets of cliques. The corresponding multiplicative scaling was used by Barvinok and Samorodnitsky~\cite{BS} to estimate perfect-matching partition
functions. Here we seek to construct a scaling with uniformly small potentials on a sparse random support.

Start with a fractional $K_r^{(k)}$-decomposition $\phi$ of the host $G$, which we call the \emph{seed}. The model theorem suggests
two useful properties of this seed: small individual weights and a
spectral gap. Let $\cK_r(G)$ be the family of vertex sets spanning
copies of $K_r^{(k)}$ in $G$, and write $E(Q)=\binom{Q}{k}$.
Define the \emph{seed Gram matrix} $A_\phi$, indexed by $E(G)$,
by letting $(A_\phi)_{ef}$ be the sum of $\phi(Q)$ over all cliques
containing both $e$ and $f$. Thus its entries record the seed weight shared by two edges. The condition $A_\phi\succeq\kappa I$ imposed below is the clique
analogue of the spectral assumption in Theorem~\ref{thm:T11}.

After retaining edges with probability $p>0$, the natural starting
weight on a retained clique is $\phi(Q)p^{1-\binom{r}{k}}$:
conditional on one edge being retained, its other $\binom{r}{k}-1$
edges all survive with probability $p^{\binom{r}{k}-1}$.
If $\phi(Q)\le R/n^{r-k}$, these starting weights are at most
$R/(n^{r-k}p^{\binom{r}{k}-1})$. The model's bound
$O(\sqrt{\beta\log N})$ therefore suggests a potential correction of order
$\sqrt{\log n/(n^{r-k}p^{\binom{r}{k}-1})}$, and hence the probability scale in Theorem~\ref{thm:minimum-degree}.
Clique contributions are, however, not independent. The theorem below accounts for this
dependence through an additional overlap term.

\Needspace{6\baselineskip}
\begin{theorem}
\label{thm:quantitative}\label{thm:inheritance}
For all integers $k\ge2$, $r\ge k+1$ and every $R,\kappa,b>0$, there are constants
$C,C',n_0>0$ such that the following holds for $n\ge n_0$.
Let $G$ be an $n$-vertex $k$-uniform hypergraph with a fractional $K_r^{(k)}$-decomposition
$\phi$ satisfying
\begin{itemize}[leftmargin=2em, itemsep=4pt, topsep=6pt]
\item \textbf{Low weight:} $\phi(Q)\le \frac{R}{n^{r-k}}$ for every $Q\in\cK_r(G)$.
\item \textbf{Spectral gap:} every eigenvalue of $A_\phi$ is at least $\kappa$.
\end{itemize}
If
$p \ge C(\log n/n^{r-k})^{1/(\binom{r}{k}-1)}$, then with probability at least $1-n^{-b}$, there is a vector
$u\in\R^{E(G_p)}$ satisfying
$\|u\|_\infty\le C'\max\Bigl\{
       \sqrt{\frac{\log n}{np^k}},\,
       \sqrt{\frac{\log n}{n^{r-k}p^{\binom{r}{k}-1}}}
       \Bigr\}$,
such that $G_p$ admits a fractional $K_r^{(k)}$-decomposition with weights
$w(Q)=\phi(Q)p^{1-\binom{r}{k}}\exp\Bigl(\sum_{e\in E(Q)}u_e\Bigr)$.
\end{theorem}

The second term in the maximum is the rate suggested by the model;
the first comes from overlaps on $k+1$ vertices.
Increasing the density constant $C$ makes both terms small, so the
resulting weights can be made arbitrarily close to the naturally
rescaled seed. More precisely, for any fixed $0<\eta<1$, choosing
$C$ sufficiently large ensures that every $w(Q)$ lies between
$(1-\eta)\phi(Q)p^{1-\binom{r}{k}}$ and
$(1+\eta)\phi(Q)p^{1-\binom{r}{k}}$.

To complete the passage from the dense host to its random
subhypergraph, we need a seed with these two properties.
The minimum-codegree assumption supplies both.

\Needspace{8\baselineskip}
\begin{theorem}
\label{thm:balanced-seed}
For all integers $k\ge2$, $r\ge k+1$ and every
$\varepsilon>0$, there are constants $\sigma,R,\kappa>0$ and
$n_0$ such that the following holds. Every $k$-uniform hypergraph
$G$ with $n\ge n_0$ vertices and
minimum codegree $\delta_{k-1}(G)\ge(\delta^*_{k,r}+\varepsilon)n$
has a fractional $K_r^{(k)}$-decomposition $\phi$ satisfying
$\frac{\sigma}{n^{r-k}}\le\phi(Q)\le \frac{R}{n^{r-k}}$ for every
$Q\in\cK_r(G)$, such that every eigenvalue of $A_\phi$
is at least $\kappa$.
\end{theorem}

Theorem~\ref{thm:minimum-degree} follows by applying
Theorem~\ref{thm:quantitative} to the seed supplied by Theorem~\ref{thm:balanced-seed}.
Our short proof of Theorem~\ref{thm:balanced-seed} in
Section~\ref{sec:seed-reduction} adapts the sampling and target-weight
adjustment underlying the beautiful theorem of Delcourt, Henderson,
Lesgourgues and Postle~\cite{DHLP} on balanced fractional
decompositions of graphs. Decompositions with one edge deleted also
supply the spectral gap. In a related setting, Kwan, Safavi and Wang~\cite{KSW} use well-distributed fractional perfect matchings to guide a random greedy process for counting perfect matchings in Dirac hypergraphs.

\medskip
\noindent\textbf{Proof outline and organization.}
Section~\ref{sec:scaling} develops the probabilistic scaling argument
for independent columns. Section~\ref{sec:proof-architecture} sets up
the clique balancing equations and reduces
Theorem~\ref{thm:quantitative} to a concentration and mean estimate.
Sections~\ref{sec:structure}--\ref{sec:random} supply the spectral
seeds, concentration tools, and random overlap estimates.
Section~\ref{sec:local} proves local stability and applies the
same deterministic response estimate, Lemma~\ref{lem:response-energy},
to clique balancing. Section~\ref{sec:bootstrap} verifies
the concentration and mean estimate, completing the proof.
Appendix~\ref{app:squared-overlap} supplies the squared-overlap count.

\subsection*{Notation}
All logarithms are natural, and $[n]=\{1,\ldots,n\}$.
Asymptotic notation refers to $n\to\infty$, or to $N\to\infty$
in the column model. For functions $f,g$ of this size parameter, with $g>0$, we write $f=O(g)$ if $|f|\le Cg$ for all sufficiently large sizes, and $f=o(g)$ if $f/g\to0$. For positive $f,g$, the notations $f\asymp g$ and $f=\Theta(g)$ mean that both $f=O(g)$
and $g=O(f)$ hold.
The positive constants $c,C$ may change from line to line;
their allowed dependence is specified locally, with subscripts
recording dependence on fixed parameters. The notation $a\ll b$
means that $a/b$ is chosen sufficiently small in terms of the fixed
parameters. An event holds \emph{with high probability} if its
probability tends to one.

For a vector $x$, let $\|x\|_1=\sum_i|x_i|$,
$\|x\|_2=(\sum_i|x_i|^2)^{1/2}$ and
$\|x\|_\infty=\max_i|x_i|$. We write
\[
 \|A\|=\sup_{\|x\|_2=1}\|Ax\|_2,
 \qquad
 \|A\|_{\infty\to\infty}
 =\sup_{\|x\|_\infty=1}\|Ax\|_\infty
 =\max_i\sum_j|A_{ij}|.
\]
For a random matrix $A$ and $1\le q<\infty$, we write
$\|A\|_{L^q(\mathrm{op})}=(\E\|A\|^q)^{1/q}$.
For symmetric matrices, $A\succeq B$ means that $A-B$ is positive
semidefinite, that is, $z^{\T}(A-B)z\ge0$ for every $z$; on a subspace,
this inequality is required only for vectors in that subspace.
We write $\ones^\perp=\{x:\ones^{\T}x=0\}$.
We use $\diag(x)$ for the diagonal matrix with diagonal $x$,
and $\diag(A)$ for the diagonal part of $A$.

For a twice differentiable function $\Phi:\R^d\to\R$, we write
$\nabla\Phi=(\partial_i\Phi)_{i=1}^d$ and
$\nabla^2\Phi=(\partial_i\partial_j\Phi)_{i,j=1}^d$ for its gradient
and Hessian. For a vector-valued map $f$, $Df$ denotes its Jacobian
matrix, with entries $(Df)_{ij}=\partial_j f_i$.

\section{Independent-column matrix scaling}\label{sec:scaling}
We prove Theorem~\ref{thm:T11} first. Independent columns allow us to
develop the concentration, mean estimate, and continuation argument
before introducing clique overlaps.

Write $p_v(x)_i=v_i e^{x_i}/(v^{\T}e^x)$, as in
Section~\ref{sec:method}; we call $p_v(x)$ the \emph{tilt} of $v$ by $x$.
A vector $x$ is centered if $\ones^{\T}x=0$; a centered Hessian or
gap refers to the restriction to $\ones^\perp$.
Put $\Pi_0=\ones\ones^{\T}/N$ and
$P_{\mathrm{pop}}=\sum_j\E[v_jv_j^{\T}]$, so $\|P_{\mathrm{pop}}-\Pi_0\|\le1-\kappa$.
The matrix $\Pi_0$ is the orthogonal projection onto constant vectors.
We use $P$ for the stochastic Gram matrices of the column model and
$\Pi$ for projections. The parameter $N$ is the matrix size,
$\beta$ is the entry bound, and $s$ is
the tail parameter in the theorem. Constants in local estimates may
depend on a fixed bound $M$ for $\|x\|_\infty$.

For a probability vector $v$, put
$\mathsf L_v=\diag(v)-vv^{\T}$. The identities
\begin{equation}\label{eq:column-conductance}
 Dp_v(x)=\mathsf L_{p_v(x)},
 \qquad
 z^{\T}\mathsf L_vz=\frac12\sum_{a,c}v_av_c(z_a-z_c)^2
\end{equation}
identify $\mathsf L_v$ as the Laplacian of the weighted complete
graph with edge weights, or \emph{conductances}, $v_av_c$.
Tilting by $x$ changes each conductance by a factor between
$e^{-4\|x\|_\infty}$ and $e^{4\|x\|_\infty}$.

We interpolate the column laws from a deterministic balanced matrix
to the given random matrix. Put $\bar v_j=\E v_j$. For
$0\le t\le1$, independently for each $j$, let $v_j(t)$ equal
$v_j$ with probability $t$ and $\bar v_j$ otherwise; the choices
are independent of the original columns. Define
\[
 \Phi_t(x)=\sum_j\log(v_j(t)^{\T}e^x)-\ones^{\T}x.
\]
A potential $x$ is \emph{stationary}, or \emph{exact}, if
$\nabla\Phi_t(x)=\sum_jp_{v_j(t)}(x)-\ones=0$, which is precisely
the row-balancing condition.
The columns remain independent, balanced in expectation, and bounded
entrywise by $\beta$. Writing $P_0=\sum_j\bar v_j\bar v_j^{\T}$,
their population Gram matrix is
\[
 P_{\mathrm{pop}}(t)=P_0+t(P_{\mathrm{pop}}-P_0)
                  \preceq P_{\mathrm{pop}},
\]
because $P_{\mathrm{pop}}-P_0$ is a sum of covariance matrices.
Thus $I-P_{\mathrm{pop}}(t)\succeq\kappa I$ on $\ones^\perp$
throughout the path.
At $t=0$ the columns are deterministically balanced, while at $t=1$
they have their original laws.

For an event $\mathcal A$, write $\one_{\mathcal A}$ for its indicator,
and let $e_i$ be a standard basis vector. A superscript $(j)$ denotes
recomputation after replacing input $j$ by an independent copy, and
$\E'_j$ integrates over that copy with the original inputs fixed.
For a symmetric matrix $\mathsf L$ with kernel $\operatorname{span}\{\ones\}$,
$\mathsf L^\dagger$ denotes its inverse on $\ones^\perp$, extended by zero on
the kernel.

\subsection{Proof strategy}
Call a column sample $(w_j)_j$ \emph{eligible} if its raw Hessian
$\sum_j\mathsf L_{w_j}$, evaluated at zero potential, is positive
definite on $\ones^\perp$.
By \eqref{eq:column-conductance}, every finite tilt of an eligible
sample has a positive centered Hessian. A centered stationary
point, when it exists, is therefore unique.

Fix the tail parameter $s$ in Theorem~\ref{thm:T11} and put
$\tau=A\sqrt{\beta s}$, where $A$ is a sufficiently large constant.
For $0\le t\le1$, set $Z_t=\min\{\tau,\|x_t\|_\infty\}$ when
the sample is eligible and has a finite centered stationary point
$x_t$, and set $Z_t=\tau$ otherwise. We will prove that
\[
 \E Z_t\le\tau/3\quad\Longrightarrow\quad\E Z_t\le\tau/8.
\]
The expected cap is continuous: it is a finite mixture over the set
of columns that have been randomized. Since $Z_0=0$, the strict
improvement keeps its expectation small all the way to $t=1$.
A tail bound for the cap then gives existence and the size of the
correction together.

The improvement combines two estimates. A replacement square sum
of order $\beta$ concentrates the potential around its mean.
Resampling a column controls the population equation at that mean,
and a population inverse bound makes the mean small. We first state
the local replacement lemma, deduce the theorem, and then prove the
lemma. The concentration tools are collected in
Section~\ref{sec:external}.

\label{sec:model-raw}
For each deterministic $t$, let $\mathcal R_t$ be the event that
$\sum_j\mathsf L_{v_j(t)}\succeq(\kappa/2)I$ on $\ones^\perp$.
Since $0\preceq\mathsf L_{v_j(t)}\preceq\beta I$ and
$\sum_j\E\mathsf L_{v_j(t)}=I-P_{\mathrm{pop}}(t)\succeq\kappa I$
on that space,
Theorem~\ref{thm:Tropp-Chernoff} gives
$\PP(\mathcal R_t^c)\le Ne^{-\kappa/(8\beta)}$, uniformly in $t$.
On $\mathcal R_t$, conductance comparison gives a fixed centered
gap in each fixed bounded box. Replacing one column decreases the
raw gap by at most $\beta$, so the same conclusion holds at every
one-column neighbor when $\beta$ is small enough.

\subsection{The local input}
\label{sec:model-local}
The next lemma constructs every one-column replacement from a bounded
exact solution. It bounds both the displacement energy
$\|x'-x\|_2^2$ of one replacement and the sum of squared responses
at a fixed coordinate. The latter estimate
will follow from the deterministic response lemma in
Section~\ref{sec:model-local-proof}, which also drives the clique proof.

\begin{lemma}[One-column stability and square sums]\label{lem:T12}
Let $V=(v_j)_j$ be column-stochastic, with entries at most $\beta$,
and suppose it has a centered exact potential $x$ with
$\|x\|_\infty\le M$. Put $P=\sum_jp_{v_j}(x)p_{v_j}(x)^{\T}$ and
suppose $I-P\succeq\kappa I$ on $\ones^\perp$.
There are constants $c,C>0$, depending only on $\kappa,M$, such that,
if $\beta\le c$, replacing any column by any probability vector with
entries at most $\beta$ gives a unique centered exact potential $x'$
satisfying
\begin{equation}\label{eq:T1}
 \|x'-x\|_\infty\le C\beta,
 \qquad \|x'-x\|_2^2\le C\beta.
\end{equation}
For independent column laws with entries at most $\beta$ and
$\sum_j\E v_j=\ones$, fresh replacements satisfy, at every original
sample obeying these hypotheses,
\begin{equation}\label{eq:T2}
 \sum_j\E'_j\big[(x_i-x_i^{(j)})^2\mid V\big]
       \le C\beta
 \qquad\text{for every coordinate }i.
\end{equation}
\end{lemma}

We apply the lemma to the interpolated columns $v_j(t)$. On
$\mathcal R_t$ and its one-column neighbors, the preceding raw-gap
comparison supplies its spectral hypothesis at every bounded
stationary point, with constants uniform in $t$.

\subsection{Proof of the independent-column theorem}
\label{sec:model-concentration}\label{sec:model-bootstrap}
\label{sec:bipboot}\label{sec:biprate}

\begin{proof}[Proof of Theorem~\ref{thm:T11}]
Fix $s$ in the stated range and use the cap
$\tau=A\sqrt{\beta s}$ from the proof strategy. We choose $A$ large
and then the theorem's constant $c$ small, so that $\tau\le1/4$
and the one-column changes $C\beta$ are at most $\tau/8$.
All local estimates below are taken in the fixed box of radius one;
their constants therefore do not depend on $s$ or $\tau$.

The eligible samples with a stationary point form an open set
relative to the column-stochastic parameter space, and that point
depends continuously on the columns there by the
implicit-function theorem. Setting the cap to $\tau$ on the
complement therefore gives a measurable statistic.
For $S\subseteq[N]$, let $a_S$ be its expected value when precisely
the columns in $S$ are drawn from their original laws and the rest
equal their means. Then
\[
 \E Z_t=\sum_{S\subseteq[N]}a_S t^{|S|}(1-t)^{N-|S|}.
\]
Thus $\E Z_t$ is continuous on $[0,1]$. At $t=0$ the mean columns
are balanced and have centered gap at least $\kappa$, so $Z_0=0$.

\proofstep{Step 1: Concentrate the cap and cutoff coordinates.}

We use Corollary~\ref{cor:poly-tail} with $h=s$, $a=1$ and $C_0=4$.
Choose a cap-tail exponent $D_Z\ge8$ larger by at least one than
the exceptional exponent needed for coordinate tails at deviation
$1/32$ with exponent $8$. Next choose a raw exceptional exponent
$D_R\ge D_Z+1$ large enough for cap tails at deviation $1/6$ with
exponent $D_Z$. The raw estimate and $\beta s\le c$ give
\[
 \PP(\mathcal R_t^c)\le Ne^{-\kappa/(8\beta)}
       \le e^{-D_Rs}
\]
when $c$ is small enough. Here $N\le e^s/2$.

On $\mathcal R_t$, the one-sided square proxy $\mathcal V_+(Z_t)$
(defined in Lemma~\ref{lem:rare-exception}) is at most $C\beta$. Indeed, if $Z_t<\tau$, fix a signed coordinate attaining
the original maximum norm. Every positive decrement of the cap is
bounded by the decrement of this same coordinate, so
Lemma~\ref{lem:T12} applies. If $Z_t=\tau$ and a replacement lowers
it, that replacement has a bounded exact solution and raw gap at
least $\kappa/2-\beta$. Applying the local lemma in reverse produces
an original solution of norm at most $\tau+C\beta$. Its maximizing
signed coordinate again bounds all positive cap decrements;
uniqueness makes this solution independent of the lowering
replacement used to construct it. If no
replacement lowers the cap, the proxy is zero. Globally the proxy
is at most $\tau^2N$.

For $Z_t/\tau$, the good-event proxy multiplied by $s$ is thus at
most $C\beta s/\tau^2=C/A^2$, and the global proxy is at most
$N\le e^s$. Increasing $A$ makes Corollary~\ref{cor:poly-tail}
applicable. At every deterministic $t\in[0,1]$ with
$\E Z_t\le\tau/3$, it gives
\begin{equation}\label{eq:model-cap-tail}
 \PP(Z_t\ge\tau/2)\le e^{-D_Zs}.
\end{equation}

Fix such a $t$, and use the cap itself as the cutoff:
\begin{equation}\label{eq:cap-cutoff-identity}
 Y=\begin{cases}x_t,&Z_t<\tau,\\0,&Z_t=\tau,\end{cases}
 \qquad
 Z_t=\|Y\|_\infty+\tau\one_{\{Z_t=\tau\}}.
\end{equation}
Thus $\|Y\|_\infty\le\tau$ on every sample. On
$\mathcal D=\mathcal R_t\cap\{Z_t<\tau/2\}$, every replacement
also lies below the cutoff, since $C\beta\le\tau/8$.
Both signed coordinate proxies of $Y$ are consequently at most
$C\beta$ there, and at most $4\tau^2N$ globally. Also
$\PP(\mathcal D^c)\le2e^{-D_Zs}\le e^{-(D_Z-1)s}$.
Lemma~\ref{lem:rare-exception} bounds the variance, and
Corollary~\ref{cor:poly-tail}, applied to $\pm Y_i/\tau$ and
followed by a union bound, gives
\begin{equation}\label{eq:T8}
 \max_i\Var Y_i\le C\beta,
 \qquad
 \PP\bigl(\|Y-\E Y\|_\infty\ge\tau/32\bigr)\le e^{-7s}.
\end{equation}
For the variance bound, the exceptional contribution is at most
$8A^2\beta sN e^{-D_Zs}=O(\beta)$. The constant may depend on the
now fixed $A$. These estimates use the unconditioned product law of
the $N$ interpolated columns at a fixed deterministic $t$.
No simultaneous good event over the path is needed.

\proofstep{Step 2: Locate the mean through the population equation.}

At this fixed $t$, write $w_j=v_j(t)$,
$g_{\mathrm{emp}}=\nabla\Phi_t$,
$g_{\mathrm{pop}}(y)=\E[g_{\mathrm{emp}}(y)]$
and $m_Y=\E Y$. The population equation at the mean splits into
three errors:
\begin{equation}\label{eq:mean-error-decomposition}
 g_{\mathrm{pop}}(m_Y)
 =\E g_{\mathrm{emp}}(Y)
   +\E\bigl[g_{\mathrm{pop}}(Y)-g_{\mathrm{emp}}(Y)\bigr]
 +g_{\mathrm{pop}}(m_Y)-\E g_{\mathrm{pop}}(Y).
\end{equation}
The cap tail, resampling, and coordinate variances will control
these three terms, respectively.

The cutoff term is supported on $\{Z_t=\tau\}$. There $Y=0$ and
$\|g_{\mathrm{emp}}(0)\|_\infty\le N\beta+1\le2N\beta$;
elsewhere $g_{\mathrm{emp}}(Y)=0$. By
\eqref{eq:model-cap-tail},
$\|\E g_{\mathrm{emp}}(Y)\|_\infty
 \le2N\beta\PP(Z_t=\tau)\le C\beta$.

For the resampling term, $Y^{(j)}$ is independent of the original
$w_j$ and has the same law as $Y$. On $\mathcal D$, the change in
$p_{w_j}(Y)_i$ is at most $C\beta(w_j)_i$; on its complement it is
at most $C(w_j)_i\le C\beta$.

Population balance therefore gives
\[
 \|\E[g_{\mathrm{emp}}(Y)-g_{\mathrm{pop}}(Y)]\|_\infty
 \le C\beta+CN\beta\PP(\mathcal D^c)\le C\beta.
\]
Combining the cutoff and resampling bounds gives
$\|\E g_{\mathrm{pop}}(Y)\|_\infty\le C\beta$.

To pass from $Y$ to its mean, direct differentiation of a tilted
column $p=p_v(y)$ gives, with $\langle p,h\rangle=\sum_a p_ah_a$,
\begin{equation}\label{eq:column-second-derivative}
 D^2p_v(y)_i[h,h]
 =p_i\biggl[(h_i-\langle p,h\rangle)^2
       -\Bigl(\sum_a p_ah_a^2-\langle p,h\rangle^2\Bigr)\biggr].
\end{equation}
Here $D^2f(y)[h,h]=\left.\frac{d^2}{ds^2}f(y+sh)\right|_{s=0}$
is the second directional derivative. The absolute value in
\eqref{eq:column-second-derivative} is at most
$C p_i(h_i^2+\sum_a p_ah_a^2)$. Bounded tilting and population
balance consequently give
\[
 |D^2g_{\mathrm{pop},i}(y)[h,h]|
 \le C\Bigl(h_i^2+\sum_a(P_{\mathrm{pop}}(t))_{ia}h_a^2\Bigr).
\]
Taylor expansion about $m_Y$ has zero expected linear term.
Since $P_{\mathrm{pop}}(t)$ is stochastic, \eqref{eq:T8} bounds its
expected remainder by $C\beta$. Hence
$\|g_{\mathrm{pop}}(m_Y)\|_\infty\le C\beta$.

The mean-value Hessian
$\mathsf L=\int_0^1Dg_{\mathrm{pop}}(u m_Y)\,du$ is a weighted
Laplacian with a fixed centered gap $h_0>0$. For $u\ge0$, the matrix
exponential $e^{-u\mathsf L}$ has nonnegative entries and row sums
one, while its operator norm on $\ones^\perp$ is at most $e^{-h_0u}$.
Consequently,
$\|e^{-u\mathsf L}-\Pi_0\|_{\infty\to\infty}
 \le\min\{2,\sqrt N e^{-h_0u}\}$.
Integrating this bound in
$\mathsf L^\dagger=\int_0^\infty(e^{-u\mathsf L}-\Pi_0)\,du$
shows that $\|\mathsf L^\dagger\|_{\infty\to\infty}\le C\log(2N)$.
Since $m_Y$ is centered and
$\mathsf Lm_Y=g_{\mathrm{pop}}(m_Y)-g_{\mathrm{pop}}(0)$, we obtain
\begin{equation}\label{eq:model-mean}
 \|m_Y\|_\infty\le C\beta\log(2N)\le\tau/64.
\end{equation}
The last inequality follows by decreasing $c$, because
$\beta\log(2N)/\tau\le\sqrt c/A$.

\proofstep{Step 3: Improve the cap and continue to the endpoint.}

Since $\|Y-m_Y\|_\infty\le2\tau$, \eqref{eq:T8} gives
$\E\|Y-m_Y\|_\infty\le\tau/32+2\tau e^{-7s}$.
Taking expectations in \eqref{eq:cap-cutoff-identity} and using
\eqref{eq:model-mean} and \eqref{eq:model-cap-tail}, we obtain
\[
 \begin{aligned}
 \E Z_t&\le\|m_Y\|_\infty+\E\|Y-m_Y\|_\infty
            +\tau\PP(Z_t=\tau)\\
 &\le\tau/64+\tau/32+2\tau e^{-7s}+\tau e^{-D_Zs}
 \le\tau/8.
 \end{aligned}
\]

Thus $\E Z_t\le\tau/3$ implies $\E Z_t\le\tau/8$.
Continuity and $\E Z_0=0$ exclude a first parameter at which the
expectation reaches $\tau/3$. The improvement and
\eqref{eq:model-cap-tail} therefore hold throughout $[0,1]$.
At $t=1$, outside an event of probability at most
$e^{-D_Zs}\le e^{-s}$, we have $Z_1<\tau/2$. By definition of the
cap, the original columns then have a centered exact potential of
norm less than $\tau/2=(A/2)\sqrt{\beta s}$, as required.
\end{proof}

\subsection{A common response estimate}
\label{sec:model-local-proof}
We isolate the deterministic estimate shared by the model and clique
arguments. The key fact is that each inverse Hessian column is close
to a coordinate direction in Euclidean norm. This controls the
linear responses after summing their squares. The Taylor errors are
quadratic, and their square sum is bounded by the individual
response energy times the square sum we seek. Small energy lets us
absorb this term. We write $v^{\circ2}$ for coordinatewise squaring.

\begin{lemma}[Square sums of nonlinear responses]\label{lem:response-energy}
For every $K\ge1$ there are constants $c_K,C_K>0$ such that the
following holds for every $\gamma>0$. Let $\mathsf H$ be a symmetric
$d\times d$ matrix, and put
$\mathsf D=\diag(\mathsf H)$. Suppose that
\begin{equation}\label{eq:energy-Hessian}
 K^{-1}I\preceq\mathsf H\preceq KI,
 \qquad
 \max_a\| (\mathsf H-\mathsf D)e_a\|_2^2\le K\gamma.
\end{equation}
Let $(f_j)_j$ be a finite nonempty family of nonnegative vectors in $\R^d$,
with weights $\omega_j\ge0$, and put
\begin{equation}\label{eq:energy-forcing}
 A_f=\sum_j\omega_j f_jf_j^{\T}.
\end{equation}
Suppose that vectors $\delta^{(j)}$ satisfy
$\max_j\|\delta^{(j)}\|_2^2\le c_K$ and admit representations
\begin{equation}\label{eq:energy-response}
 \delta^{(j)}=\mathsf H^{-1}(g_j+r_j),
 \qquad |g_j|\le Kf_j,
 \qquad |r_j|\le K\mathsf B(\delta^{(j)})^{\circ2},
\end{equation}
where $\mathsf B$ is an entrywise nonnegative matrix with
$\|\mathsf B\|\le K$, and the two inequalities are coordinatewise.
Then
\[
 \max_a\sum_j\omega_j(\delta_a^{(j)})^2
 \le C_K\bigl(\max_a(A_f)_{aa}+\gamma\|A_f\|\bigr).
\]
In particular, for one response with unit weight,
\begin{equation}\label{eq:single-response}
 \|\delta\|_\infty^2
 \le C_K\bigl(\|f\|_\infty^2+\gamma\|f\|_2^2\bigr).
\end{equation}
The conclusion remains valid for measurable families indexed by a finite measure space, with weighted sums replaced by integrals and maxima over $j$ by essential suprema, provided the forcing second moments are finite.
\end{lemma}

\begin{proof}
All constants in this proof depend only on $K$. For a coordinate $a$,
write $z_a=\mathsf H^{-1}e_a$. The identity
\begin{equation}\label{eq:inverse-column}
 z_a=\mathsf D_{aa}^{-1}e_a
       -\mathsf D_{aa}^{-1}\mathsf H^{-1}
          (\mathsf H-\mathsf D)e_a
     =\mathsf D_{aa}^{-1}e_a+\xi_a
\end{equation}
gives $\|\xi_a\|_2^2\le C\gamma$. Since $A_f$ is entrywise
nonnegative, $|z_a|\le\mathsf D_{aa}^{-1}e_a+|\xi_a|$ implies
\begin{equation}\label{eq:linear-square}
 \sum_j\omega_j\langle |z_a|,f_j\rangle^2
 =|z_a|^{\T}A_f|z_a|
 \le2\mathsf D_{aa}^{-2}(A_f)_{aa}
        +2\|A_f\|\|\xi_a\|_2^2
 \le C\bigl((A_f)_{aa}+\gamma\|A_f\|\bigr).
\end{equation}

Put $\mathcal V=\max_a\sum_j\omega_j(\delta_a^{(j)})^2$ and
$E=\max_j\|\delta^{(j)}\|_2^2$. For every fixed vector $h$,
Cauchy--Schwarz followed by summation gives
\begin{equation}\label{eq:energy-CS}
 \begin{aligned}
 \sum_j\omega_j\langle h,(\delta^{(j)})^{\circ2}\rangle^2
 &\le E\sum_a h_a^2\sum_j\omega_j(\delta_a^{(j)})^2\\
 &\le E\mathcal V\|h\|_2^2.
 \end{aligned}
\end{equation}
The response representation gives
\[
 |\delta_a^{(j)}|
 \le K\langle |z_a|,f_j\rangle
       +K\langle \mathsf B^{\T}|z_a|,(\delta^{(j)})^{\circ2}\rangle.
\]
Here $\|\mathsf B^{\T}|z_a|\|_2\le\|\mathsf B\|\|z_a\|_2\le K^2$.
Squaring, summing over $j$, and applying
\eqref{eq:linear-square}--\eqref{eq:energy-CS} therefore yields
\[
 \mathcal V\le
 C_1\bigl(\max_a(A_f)_{aa}+\gamma\|A_f\|\bigr)+C_2 E\mathcal V.
\]
The quantity $\mathcal V$ is finite: each coordinate square is at
most $E$, and the total weight or measure is finite. Since
$E\le c_K$, choosing $C_2c_K\le1/2$ proves the lemma. In the integral
case the same argument uses Tonelli's theorem in
\eqref{eq:energy-CS}.
\end{proof}

The two forcing terms control the linear response at a coordinate
and the contribution from the off-diagonal part of the Hessian.
The individual energy $E$ controls the nonlinear absorption.
The lemma applies both to one replacement and to their whole family.

\begin{proof}[Proof of Lemma~\ref{lem:T12}]
Write $p_j(y)=p_{v_j}(x+y)$ and
$p_j=p_j(0)$. For a fresh column $v'$, write
$u(y)=p_{v'}(x+y)$ and $u=u(0)$. All population expectations
in this proof hold $x$ fixed. Put $b=e^{2M}\beta$ and
$P=\sum_jp_jp_j^{\T}$.
At the original stationary point, the row loads are one.
Thus $P$ is symmetric, stochastic and positive semidefinite, with
diagonal at most $b$, and the Hessian is $I-P$.
Its centered gap is at least $\kappa$ by hypothesis. Here and below
constants may depend on $\kappa,M$.

\proofstep{Step 1: Find the new solution by local convexity.}

Let $\widetilde\Phi(y)$ be the replaced objective at $x+y$,
restricted to $\ones^\perp$. Its gradient at zero is $u-p_j$,
of Euclidean norm $O(\sqrt\beta)$. Its Hessian there loses
$\mathsf L_{p_j}\preceq bI$ and gains $\mathsf L_u\succeq0$.
By \eqref{eq:column-conductance}, its centered gap is at least
$\kappa/2$ throughout $\|y\|_\infty\le1/16$, for small enough
$\beta$. Consequently
\[
 \widetilde\Phi(y)-\widetilde\Phi(0)
 \ge -C\sqrt\beta\,\|y\|_2+\frac\kappa4\|y\|_2^2.
\]
On the boundary of a Euclidean ball in $\ones^\perp$ of radius
$C_1\sqrt\beta$, with $C_1$ sufficiently large, this is positive.
For small $\beta$ the ball lies in the working box, so its minimizer
$\delta$ is interior and stationary. The full gradient is centered,
so it vanishes there. Thus the new exact potential exists and
$\|\delta\|_2^2\le C\beta$.
Conductance comparison gives a positive centered Hessian at every
finite point, hence global uniqueness. The implicit-function theorem
makes the solution continuous in the replacing column and therefore
measurable.

\proofstep{Step 2: Use the original Hessian for every replacement.}

Let $G(y)=\sum_jp_j(y)-\ones$ and set $\mathsf H=I-P+\Pi_0$.
For $\delta=x'-x$, put
$r(\delta)=G(\delta)-G(0)-(I-P)\delta$.
The replacement equation and $\delta\perp\ones$ give
\[
 \delta=\mathsf H^{-1}
       \bigl(p_j(\delta)-u(\delta)-r(\delta)\bigr).
\]
Bounded tilting gives
$|p_j(\delta)-u(\delta)|\le C f_j$, where
$f_j=p_j+u\ge0$.
The second-derivative formula \eqref{eq:column-second-derivative},
summed over the original columns, gives the Taylor bound
below; bounded tilting and row balance at $y=0$ control its coefficients:
\[
 |r(\delta)|\le C(I+P)\delta^{\circ2}.
\]
The envelope $\mathsf B=I+P$ is entrywise nonnegative and has
operator norm at most two. The original $\mathsf H$ satisfies
\eqref{eq:energy-Hessian} with $\gamma=\beta$: indeed,
$\|Pe_i\|_2^2\le P_{ii}\le b$ and $1/N\le\beta$.
For one replacement, $f_j=p_j+u$ has maximum norm $O(\beta)$
and squared Euclidean norm $O(\beta)$. The energy bound from
Step~1 is small enough for Lemma~\ref{lem:response-energy}.
Its single-response estimate \eqref{eq:single-response}, with
$\gamma=\beta$, gives $\|\delta\|_\infty^2\le C\beta^2$,
proving the remaining bound in \eqref{eq:T1}.

\proofstep{Step 3: Sum the squared responses.}

For fresh tilted columns, let
$Q=\sum_j\E'_j[p_{v'_j}(x)p_{v'_j}(x)^{\T}]$.
Bounded tilting gives $Q\le e^{4M}P_{\mathrm{pop}}$ entrywise.
Population balance makes $P_{\mathrm{pop}}$ stochastic, with
diagonal at most $\beta$. Thus $Q$ has row sums at most $e^{4M}$,
so $\|Q\|\le e^{4M}$ and
$\max_i Q_{ii}\le e^{4M}\beta$.
For the forcing family,
\[
 A_f:=\sum_j\E'_j[f_jf_j^{\T}]
 \preceq2P+2Q.
\]
Thus $\|A_f\|\le C$ and $\max_i(A_f)_{ii}\le C\beta$.
All responses have energy $O(\beta)$ by Step~1. Apply
Lemma~\ref{lem:response-energy} with $\gamma=\beta$,
$\mathsf B=I+P$.
The replacement index consists of the column index together with
its fresh input; its measure is the sum of the $N$ fresh column
laws. The integral form of the lemma gives \eqref{eq:T2}.
\end{proof}

\section{From the model to clique inheritance}
\label{sec:proof-architecture}
The clique problem has two additional features: one edge indicator
affects many cliques, and changing it changes the potential's
coordinate set. Squared-overlap estimates control the first effect;
virtual coordinates resolve the second. This section defines these
objects and reduces Theorem~\ref{thm:quantitative} to a strict
improvement of an expected cap.

The proof follows the model with edge indicators in place of columns
and $\gamma_{k,r}$ in place of the square-sum scale $\beta$.
With a cap at scale $\sqrt{\gamma_{k,r}\log n}$, we use the same
three steps: concentrate the cutoff potential, locate its mean
through the population equation, and continue the expected cap from
the full host to the target density. Both proofs vary the product
law from a deterministic balanced starting point.

The clique population Hessian has a uniformly bounded inverse in
the maximum norm (Section~\ref{sec:inverse}). Consequently the individual bound
$O(\sqrt{\gamma_{k,r}})$ supplied by the energy estimate is enough
to locate the mean. In the model, the population inverse costs
$\log N$, which is why its local lemma also proves the sharper
$O(\beta)$ individual bound.

We retain the norm and indicator notation from the model. For an
edge vector $v$ and an edge set $S$, write
$v(S)=\sum_{e\in S}v_e$. We write $Q\ni e$ for $e\in E(Q)$.
For $Q\in\cK_r(G)$, let $\ones_Q\in\R^{E(G)}$ be the indicator of
$E(Q)$. The seed matrix from Section~\ref{sec:seed-intro} then has
the Gram representation
$A_\phi=\sum_{Q\in\cK_r(G)}\phi(Q)\ones_Q\ones_Q^{\T}$.
For $Q\in\cK_r(F)$ with $F\subseteq G$, the same symbol $\ones_Q$
denotes its restriction to $E(F)$.

Throughout the clique proof, $G$ and $\phi$ satisfy the seed
assumptions of Theorem~\ref{thm:quantitative}. An empty host has the
empty fractional decomposition, so we henceforth assume $E(G)$ is
nonempty. In particular, $0<\kappa\le1$, since $A_\phi$ has diagonal
one. Constants may depend on $k,r,R,\kappa$ and fixed tail or box
parameters, but not on $n$, the density, or the chosen host and seed.

\subsection{The error parameter}\label{sec:quantitative}
For the proof, write
$\mu_s(n,p)=n^{s-k}p^{\binom{s}{k}-1}$ for the extension scale associated
with $K_s^{(k)}$, where $k+1\le s\le r$. The error parameter is the
largest reciprocal of these scales and depends only on the two
endpoints:
\begin{equation}\label{eq:gamma-endpoints}
 \gamma_{k,r}(n,p)=\max_{k+1\le s\le r}\frac1{\mu_s(n,p)}
 =\max\Bigl\{\frac1{np^k},\frac1{\mu_r(n,p)}\Bigr\}.
\end{equation}
The endpoint identity follows from discrete concavity of
$(s-k)\log n+(\binom{s}{k}-1)\log p$.
Thus the maximum in Theorem~\ref{thm:quantitative} is exactly
$\sqrt{\gamma_{k,r}(n,p)\log n}$. It suffices to prove the theorem
under $\mu_r(n,p)\ge C_{\mathrm{ext}}\log n$, taking
$C=C_{\mathrm{ext}}^{1/(\binom{r}{k}-1)}$ in its statement. Throughout the proof,
$C_{\mathrm{ext}}$ denotes this extension-scale constant, rather than the coefficient
of the lower bound on $p$.

Put $a=(r-k)/(\binom{r}{k}-1)$. Since the successive differences of $\binom{s}{k}-1$ are strictly
increasing for $s\ge k$, we have $a\le1/k$ and, for every integer
$s$ with $k+1\le s<r$,
\begin{equation}\label{eq:strict-density}
 s-k-a\Bigl[\binom{s}{k}-1\Bigr]>0.
\end{equation}
Equality in $a\le1/k$ holds only when $r=k+1$. In particular, the
extension condition implies $p\ge n^{-a}$ for sufficiently large $n$.

\subsection{The scale of replacement changes}
We abbreviate $\mu=\mu_r(n,p)$ and $\gamma=\gamma_{k,r}(n,p)$ when
the density is clear, and use the corresponding quantities at $q$
when another density is specified. The extension condition and
\eqref{eq:strict-density} give, uniformly in $p$,
\begin{equation}\label{eq:gamma-small}
 \gamma_{k,r}(n,p)\log n\le C_{\mathrm{ext}}^{-1}+o(1).
\end{equation}
Thus $\gamma\log n$ can be made arbitrarily small by increasing
$C_{\mathrm{ext}}$ and then $n$. The scale $\gamma$ is nonincreasing
in the density.

A single edge replacement changes each extended potential coordinate
by at most $O(\sqrt\gamma)$. Local convexity supplies this bound
through a squared Euclidean displacement estimate $O(\gamma)$ on
the original retained edges. More importantly, the sum of squared
changes of any fixed coordinate over all replacements is also
$O(\gamma)$.
This second estimate gives concentration at the scale
$\sqrt{\gamma\log n}$. It must be proved directly: combining an
individual overlap bound with a first-moment bound loses the
required scale, as Remark~\ref{rem:max-obstruction} explains.

\subsection{The balancing equations on a fixed coordinate space}
\label{sec:potential-setup}
For a density $q\in(0,1]$, a subhypergraph $F\subseteq G$ and
$v\in\R^{E(F)}$, put
\[
 \Phi_{F,q}(v)=q^{1-\binom{r}{k}}\sum_{Q\in\cK_r(F)}\phi(Q)e^{v(E(Q))}
                       -\sum_{e\in E(F)}v_e.
\]
Write $\mathsf H_F(v)=\nabla^2\Phi_{F,q}(v)$ and suppress the fixed
parameter $q$ in $\Phi_F=\Phi_{F,q}$. For $Q\in\cK_r(F)$, put
$w_Q(v)=q^{1-\binom{r}{k}}\phi(Q)e^{v(E(Q))}$.
We call the loads and Hessian at $v=0$ \emph{raw}, keeping the
factor $q^{1-\binom{r}{k}}$ in the weights.
As in the model, $v$ is stationary, or exact, if
$\nabla\Phi_{F,q}(v)=0$: the weights $w_Q(v)$ then give every
edge load one and hence form a fractional decomposition of $F$.
Unlike column scaling, this problem has no
additive normalization freedom: adding a constant $c$ to the potential
multiplies all clique weights by $e^{\binom{r}{k}c}$. We therefore impose no
centering condition on a clique potential.
To compare different edge sets, we extend a potential to every host
edge. For an absent edge $e$, let
\[
 s_e(v)=\sum_{\substack{Q\in\cK_r(G),\ e\in E(Q)\\
                 E(Q)\setminus\{e\}\subseteq E(F)}}
           \phi(Q)e^{v(E(Q)\setminus\{e\})}.
\]
If all these completion sums are positive, define the extended potential by
\begin{equation}\label{eq:virtual-definition}
 \widehat v_e=
 \begin{cases}
 v_e,&\text{if }e\in E(F),\\
 \log(q^{\binom{r}{k}-1}/s_e(v)),&\text{if }e\in E(G)\setminus E(F).
 \end{cases}
\end{equation}
We call the second value a \emph{virtual coordinate}. It is the value
that would balance edge $e$ if that edge were added while the present
coordinates were held fixed.

For a stationary potential the extension satisfies the edge equation
also at absent host edges. This supplies a fixed coordinate space
$\R^{E(G)}$ on which expectations and replacements can be compared.
If a hypergraph has a zero completion sum, the capped statistic below is
defined without using an extended potential.

\subsection{A capped statistic and the existence argument}
\label{sec:cap-setup}
We use the model's cap construction on the fixed host-edge space.
The cap is defined on every hypergraph, including those whose
balancing equations have no finite solution. As in the model, only
the expectation of the cap needs to vary continuously.

Fix a target density $p$ in the range of Theorem~\ref{thm:quantitative}. Let independent uniform $[0,1]$ variables $U_e$, indexed by all $k$-sets, generate the coupled hypergraphs $G(q)=G\cap\{e:U_e\le q\}$ for $q\in[p,1]$. Set
$\tau_n=C_{\mathrm{cap}}\sqrt{\gamma_{k,r}(n,p)\log n}$,
where $C_{\mathrm{cap}}>0$ is a sufficiently large constant, chosen before the density
constant $C_{\mathrm{ext}}$ and the lower bound $n_0$. This cap remains fixed while $q$ varies.

Call a fixed subhypergraph $F\subseteq G$ \emph{eligible} if each
host edge $e$ lies in a clique $Q$ with $\phi(Q)>0$ and
$E(Q)\setminus\{e\}\subseteq E(F)$, and the matrix
$\sum_{Q\in\cK_r(F)}\phi(Q)\ones_Q\ones_Q^{\T}$ is positive definite
on $\R^{E(F)}$. Only cliques with $\phi(Q)>0$ contribute to this
matrix. When $F$ is eligible and a finite stationary potential $u_{F,q}$ exists, put $Z_F(q)=\min\{\tau_n,\|\widehat u_{F,q}\|_\infty\}$; otherwise put $Z_F(q)=\tau_n$. Uniqueness of the potential in this definition is included in the next lemma. Write $Z(q)=Z_{G(q)}(q)$ and $\bar Z(q)=\E Z(q)$. Thus $Z_F(q)$ is the cap for a fixed hypergraph, $Z(q)$ is its random
value, and $\bar Z(q)$ is its expectation.

\begin{lemma}\label{lem:cap-continuity}
The capped statistic is well defined on every sample. The function $\bar Z$ is continuous on $[p,1]$, and $\bar Z(1)=0$. For every $q\in[p,1]$, if $Z(q)<\tau_n$, then $G(q)$ has a unique finite stationary potential $u$ satisfying
$\|u\|_\infty\le\|\widehat u\|_\infty=Z(q)$.
\end{lemma}
\begin{proof}
The kernel consists of vectors orthogonal to every supported clique
vector $\ones_Q$, independently of the positive weights.
Thus every positive tilt on an eligible hypergraph has a positive definite Hessian. Its objective is strictly convex, and a finite stationary potential, when it exists, is globally unique. Eligibility also makes every virtual coordinate finite. Setting the cap to $\tau_n$ on all remaining samples therefore defines a bounded statistic on the entire product space, without assuming existence on a good event.

For a fixed eligible hypergraph, existence is independent of the
positive normalization parameter. Indeed, changing $q$ to $q'$ adds
$\frac{\binom{r}{k}-1}{\binom{r}{k}}\log(q'/q)\one$ to both the actual
and the extended stationary potential. For the actual potential this
shift preserves every clique weight. Substituting it into the virtual
coordinate formula gives the same shift there, because a completion
sum contains $\binom{r}{k}-1$ side coordinates. Thus $Z_F(q)$ is
continuous. If $F$ is ineligible or has no finite stationary potential,
its cap is constantly $\tau_n$.

For fixed $n$, the expectation $\bar Z(q)$ is a finite sum over subhypergraphs $F\subseteq G$ of these continuous functions times $q^{|E(F)|}(1-q)^{|E(G)|-|E(F)|}$. It is therefore continuous on $[p,1]$. At $q=1$, the seed balance gives the stationary point $u=0$, and the seed gap makes $G$ eligible. Thus $\bar Z(1)=0$. Nonhost uniforms do not affect the statistic.

Finally, $Z(q)<\tau_n$ excludes the cases in which the cap was set to $\tau_n$ and also excludes truncation of its norm. The stationary potential exists and is unique, and its restriction to present edges has norm at most that of its extension. This proves the last assertion.
\end{proof}

On the fixed host-edge space, define $Y(q)=\widehat u_{G(q),q}$
if $Z(q)<\tau_n$, and set $Y(q)=0$ when $Z(q)=\tau_n$.
Thus $\|Y(q)\|_\infty\le\tau_n$, and the exact identity
\eqref{eq:cap-cutoff-identity} holds with $Z(q),Y(q),\tau_n$ in
place of $Z_t,Y,\tau$. The next proposition, proved in
Section~\ref{sec:bootstrap}, is the only probabilistic input
needed for the existence argument.

\Needspace{6\baselineskip}
\begin{proposition}[Cap improvement]\label{prop:bootstrap-inputs}
For the fixed parameters in Theorem~\ref{thm:quantitative}, choose
$C_{\mathrm{cap}}$ sufficiently large and then $C_{\mathrm{ext}},n_0$ sufficiently large. For every
deterministic $q\in[p,1]$, if $\bar Z(q)\le \tau_n/3$, then
\[
 \PP(Z(q)\ge\tau_n/2)\le n^{-b},
 \qquad \bar Z(q)\le\tau_n/8.
\]
\end{proposition}

\begin{proof}[Proof of Theorem~\ref{thm:quantitative}]
By Lemma~\ref{lem:cap-continuity}, $\bar Z$ is continuous and
$\bar Z(1)=0$. It cannot take the value $\tau_n/3$, since
Proposition~\ref{prop:bootstrap-inputs} would then give
$\bar Z(q)\le\tau_n/8$. Hence $\bar Z(q)<\tau_n/3$ throughout
$[p,1]$.

Apply the same proposition at $q=p$. Outside an event of probability
at most $n^{-b}$, we have $Z(p)<\tau_n/2$. By
Lemma~\ref{lem:cap-continuity}, a unique finite stationary potential
then exists, and
$\|u\|_\infty\le\|\widehat u\|_\infty=Z(p) <\frac{C_{\mathrm{cap}}}{2}\sqrt{\gamma_{k,r}(n,p)\log n}$.
Stationarity of $\Phi_{G_p,p}$ means precisely that the weights in
Theorem~\ref{thm:quantitative} balance every retained edge. They are
nonnegative, and positive on the retained seed support. Taking
$C'=C_{\mathrm{cap}}/2$ and the density constant $C$ as in
Section~\ref{sec:quantitative} proves the theorem.
\end{proof}

\subsection{The three inputs to cap improvement}
Proposition~\ref{prop:bootstrap-inputs} rests on three estimates.

\proofstep{Random stability.}
Lemma~\ref{lem:robust} supplies nearly balanced raw loads, a fixed
Hessian gap, and squared-overlap bounds at scale $\gamma_{k,r}$. Its local bounds survive the bounded
number of forced flips needed later. The replacement-family bounds
are needed only at the original hypergraph.

\proofstep{Squared responses.}
Proposition~\ref{prop:one-edge} constructs every replacement from an
existing small potential and bounds its displacement energy by
$O(\gamma_{k,r})$. For the displacements $\delta^{(j)}$ on the original present edges,
including the virtual value at a deleted edge, and flip probabilities $\omega_j$, put
$\mathcal V=\max_a\sum_j\omega_j(\delta_a^{(j)})^2$.
Lemma~\ref{lem:response-energy} compares all these responses using
the original Hessian. Their linear parts contribute $O(\gamma_{k,r})$;
their Taylor remainders contribute $O(\gamma_{k,r}\mathcal V)$.
Thus
\[
 \mathcal V\le C_1\gamma_{k,r}+C_2\gamma_{k,r}\mathcal V
       \quad\Longrightarrow\quad\mathcal V=O(\gamma_{k,r}).
\]
Proposition~\ref{prop:square} verifies the lemma's hypotheses and
extends the bound to virtual coordinates.

\proofstep{Concentration and the mean.}
The response bound gives concentration for the capped norm and
cutoff coordinates on the original product space. To locate the
mean, resample all side edges of one clique: the recomputed potential
is independent of their original indicator product. The resulting
population error is $O(\sqrt{\gamma_{k,r}})$, and coordinate variances bound
the Taylor error at the mean by $O(\gamma_{k,r})$.
Corollary~\ref{cor:population-inverse} then gives
$\|\E Y(q)\|_\infty=O(\sqrt{\gamma_{k,r}(n,q)})=o(\tau_n)$.
Concentration around this small mean improves the expected cap.
Section~\ref{sec:bootstrap} verifies these estimates, including the
exceptional events and the order of constants.

\section{Seeds and deterministic stability}\label{sec:structure}
We first prove Theorem~\ref{thm:balanced-seed}, which supplies
balanced spectral seeds above the dense fractional threshold.
We then establish the two deterministic stability estimates used
in the probabilistic argument: the seed gap survives random
coordinate restriction, and every boundedly tilted population
Hessian has a dimension-free inverse bound in the maximum norm.

\subsection{Seeds above the minimum-codegree threshold}
\label{sec:seed-reduction}
The proof of Theorem~\ref{thm:balanced-seed} has three steps.
We sample bounded-order induced hosts, use decompositions with one
edge deleted to prescribe nearby edge loads and obtain a local gap,
and average the resulting weightings with exactly corrected loads.
The sampling and target-weight adjustment follow the strategy used
for graphs by Delcourt, Henderson, Lesgourgues and
Postle~\cite[Section~4]{DHLP}. The additional spectral observation is
that single-edge-deletion decompositions provide a bounded
certificate for every edge-coordinate vector. No large-clique
extension hypothesis is required of the dense host.

All parameters $k,r,\varepsilon$ are fixed as $n$ tends to infinity.
No estimate for the numerical value of $\delta^*_{k,r}$ is needed
in the proof.

\begin{lemma}\label{lem:dense-sampling}
Fix $k,r$ and $\varepsilon>0$. For all sufficiently large integers
$\ell$, and then all sufficiently large $n$, the following holds
uniformly in $0\le d\le1-\varepsilon$. If
$\delta_{k-1}(G)\ge(d+\varepsilon)n$ and $S$ is a uniformly random
$\ell$-set containing any prescribed set $R$ of at most $r$ vertices,
then $\delta_{k-1}(G[S])\ge(d+\varepsilon/2)\ell$ with probability
at least $1-\eta_\ell$, where
$\eta_\ell=2\ell^{k-1}e^{-\varepsilon^2\ell/32}$.
In particular, $\eta_\ell\le[2\binom{\ell}{k}]^{-1}$ for sufficiently
large $\ell$.
\end{lemma}
\begin{proof}
List the prescribed vertices first and sample the remaining vertices without replacement into labeled positions. Fix any $k-1$ positions and condition on the actual vertices $T$ in those positions. Together with $R$, these fix a set $U$ of at most $r+k-1$ vertices. The remaining $\ell-|U|$ sampled vertices form a uniform subset of $V(G)\setminus U$.

The set $T$ has at least $(d+\varepsilon)n-|U|$ neighbors outside $U$.
The number $X$ of those neighbors in the remaining sample is hypergeometric. For $n\ge\ell$ its mean is at least $(d+\varepsilon)\ell-2|U|$, and hence at least $(d+3\varepsilon/4)\ell$ once $\ell\ge8(r+k-1)/\varepsilon$.
The codegree of $T$ in $G[S]$ is at least $X$. Hoeffding's bound for sampling without replacement~\cite[Theorems~2 and~4]{Hoeffding} therefore gives
\[
 \PP\bigl(d_{G[S]}(T)<(d+\varepsilon/2)\ell\mid U\bigr)
 \le e^{-\varepsilon^2\ell/8}.
\]
The bound is uniform in the conditioned vertices. Union-bound over the at most $\binom\ell{k-1}$ choices of positions, not over all $(k-1)$-sets in the $n$-vertex host. This proves the asserted, slightly weaker, bound. Its exponential decay dominates the fixed polynomial $\binom{\ell}{k}$.
\end{proof}

For a $k$-uniform hypergraph $J$, write $B_J$ for its edge--clique incidence matrix: the column indexed by $Q\in\cK_r(J)$ is $\one_Q$.
For a clique-weight vector $\psi$, its edge-load vector is $B_J\psi$
and its Gram matrix is $A_\psi=B_J\diag(\psi)B_J^{\T}$.
A coordinate vector in the clique space is denoted by $\mathbf e_Q$; a coordinate vector in the edge space is denoted by $\mathbf e_e$.

\begin{lemma}\label{lem:deletion-seed}
Let $J$ be a nonempty $k$-uniform hypergraph on at most $\ell$ vertices. Suppose
that $J$ and every $J-e$, $e\in E(J)$, have fractional
$K_r^{(k)}$-decompositions. Then $J$ has a fractional decomposition
$\alpha$ with $[\binom{r}{k}\binom{\ell}{r}]^{-1}\le\alpha(Q)\le1$
for every $Q\in\cK_r(J)$.

Moreover, if $0\le\eta\le[2\binom{\ell}{k}]^{-1}$ and the target
edge loads satisfy $a_e\in[1-\eta,1]$, there is a nonnegative clique
weighting $\psi$ with $B_J\psi=a$, every weight between
$[2\binom{r}{k}\binom{\ell}{r}]^{-1}$ and $1$, and
$A_\psi\succeq\kappa_\ell I$, where
$\kappa_\ell=[2\binom{r}{k}\binom{\ell}{k}\binom{\ell}{r}^{2}]^{-1}$.
\end{lemma}

\begin{proof}
Decompositions with one edge deleted provide the three ingredients:
positive clique weights, adjustments of the target loads, and
certificates for the spectral gap.

For each $e\in E(J)$, choose a fractional decomposition $\beta_e$
of $J-e$ and extend it by zero to the cliques of $J$. Then
$B_J\beta_e=\one-\mathbf e_e$, and all entries of $\beta_e$
lie in $[0,1]$. The clique family is nonempty because $J$ is
nonempty and fractionally decomposable.

For each fixed clique $Q$, the vector $\frac1{\binom{r}{k}}
 \left(\mathbf e_Q+\sum_{e\in E(Q)}\beta_e\right)$
is a fractional decomposition of $J$: its edge loads are
$\frac1{\binom{r}{k}}(\one_Q+\binom{r}{k}\one-\one_Q)=\one$.
Average these vectors over all $Q\in\cK_r(J)$ to obtain $\alpha$.
The explicit $\mathbf e_Q$ term gives
$\alpha(Q)\ge[\binom{r}{k}\,|\cK_r(J)|]^{-1}$.
The upper bound follows from nonnegativity and edge loads equal to one.

Set $\lambda_e=1-a_e$ and $\Lambda=\sum_e\lambda_e$.
Our choice of $\eta$ gives $0\le\Lambda\le1/2$.
Define
$\psi=(1-\Lambda)\alpha+\sum_{e\in E(J)}\lambda_e\beta_e$.
This is a convex combination of nonnegative vectors with entries at most one. Its edge-load vector is
$(1-\Lambda)\one+\sum_e\lambda_e(\one-\mathbf e_e)=a$.
Also $\psi\ge\alpha/2$ entrywise, proving the lower weight bound.

For each edge $e$, the signed vector $\alpha-\beta_e$ satisfies
\[
 B_J(\alpha-\beta_e)=\mathbf e_e,
 \qquad
 \norm{\alpha-\beta_e}_2^2\le |\cK_r(J)|\le\binom{\ell}{r}.
\]
For any edge vector $x$, Cauchy--Schwarz gives
$x_e^2\le\binom{\ell}{r}\,\norm{B_J^{\T}x}_2^2$.
Summing over at most $\binom{\ell}{k}$ edges yields
$B_JB_J^{\T}\succeq \frac1{\binom{\ell}{k}\binom{\ell}{r}}I$.
Since every $\psi(Q)$ is at least $[2\binom{r}{k}\binom{\ell}{r}]^{-1}$, termwise positive semidefinite comparison proves the claimed gap.
In particular, the gap comes from edge-deletion feasibility, not from a supposition that positive clique weights automatically imply full rank.
\end{proof}

\begin{proof}[Proof of Theorem~\ref{thm:balanced-seed}]
Write $d=\delta^*_{k,r}$. If $d+\varepsilon\ge1$, there is no such sufficiently large host; assume $d+\varepsilon<1$.
Choose a fixed integer $\ell\ge r$ large enough for Lemma~\ref{lem:dense-sampling}, with $\eta:=\eta_\ell\le[2\binom{\ell}{k}]^{-1}$.
Increase $\ell$ so that the definition of $d$ guarantees a fractional decomposition of every $\ell$-vertex $k$-uniform hypergraph of minimum codegree at least $(d+\varepsilon/4)\ell$, and so that $\varepsilon\ell/4\ge1$.

Let $\mathcal S$ be the family of $\ell$-sets $S$ such that
$\delta_{k-1}(G[S])\ge(d+\varepsilon/2)\ell$.
For $S\in\mathcal S$, the hypergraph $J=G[S]$ and each $J-e$ meet the threshold $(d+\varepsilon/4)\ell$: deleting one hyperedge lowers every codegree by at most one.
Thus Lemma~\ref{lem:deletion-seed} applies to every such $J$.

Set $M=\binom{n-k}{\ell-k}$, and let $M_e$ count the sets
$S\in\mathcal S$ containing $e$.
Lemma~\ref{lem:dense-sampling}, with the vertices of $e$ prescribed, gives
$(1-\eta)M\le M_e\le M$.
Choose the target for edge $e$, on every good set containing it, to be
$a_e=(1-\eta)M/M_e\in[1-\eta,1]$.
For each $S\in\mathcal S$, use the local lemma to obtain a weighting $\psi_S$ of the cliques of $G[S]$ with these target loads. Extend it by zero to all host cliques and define
\begin{equation*}
 \phi(Q)=\frac1{(1-\eta)M}\sum_{S\in\mathcal S}\psi_S(Q).
\end{equation*}
Every edge now has load exactly one, since
$\sum_{Q\ni e}\phi(Q)=M_ea_e/((1-\eta)M)=1$.
This explicit correction is necessary: an uncorrected average would have edge loads $M_e/M$, which need not be equal.

Fix a host clique $Q$. Applying the sampling lemma with its $r$ vertices prescribed shows that it belongs to at least
$(1-\eta)\binom{n-r}{\ell-r}$ good sets, and trivially to at most $\binom{n-r}{\ell-r}$.
The local weight bounds therefore give
\[
 \phi(Q)\ge
 \frac1{2\binom{r}{k}\binom{\ell}{r}}
 \frac{\binom{n-r}{\ell-r}}{\binom{n-k}{\ell-k}}
 =\frac1{2\binom{\ell}{k}\binom{n-k}{r-k}},\qquad
 \phi(Q)\le
 \frac1{1-\eta}
 \frac{\binom{n-r}{\ell-r}}{\binom{n-k}{\ell-k}}
 \le\frac{2\binom{\ell-k}{r-k}}{\binom{n-k}{r-k}}.
\]
Here we used
$\binom{n-r}{\ell-r}/\binom{n-k}{\ell-k}
 =\binom{\ell-k}{r-k}/\binom{n-k}{r-k}$.
Since $\ell,k,r$ are fixed, these are the required two-sided bounds of order $n^{k-r}$.

Finally, embed the local Gram-matrix bounds from
Lemma~\ref{lem:deletion-seed} in the full host-edge space. Let $\Pi_S$ be the diagonal projection onto $E(G[S])$.
Then
\[
 A_\phi
 \succeq\frac{\kappa_\ell}{(1-\eta)M}\sum_{S\in\mathcal S}\Pi_S
 =\kappa_\ell\,\diag\Bigl(\frac{M_e}{(1-\eta)M}:e\in E(G)\Bigr)
 \succeq\kappa_\ell I.
\]
This proves the seed gap with $\kappa=\kappa_\ell$ and completes the proof.
\end{proof}

\begin{remark}\label{rem:uniform-seed}
For the complete host, the uniform seed has a direct gap proof.
Let $d_n=\binom{n-k}{r-k}$ and set $\phi(Q)=d_n^{-1}$.
For $0\le j\le k$, let $\mathsf W_j$ have rows indexed by $j$-sets, columns indexed by $k$-sets, and entry $1$ for containment. Then
\begin{equation*}
 A_\phi=\frac1{d_n}\sum_{j=0}^k
       \binom{n-2k}{r-2k+j}\mathsf W_j^{\T}\mathsf W_j,
 \qquad
 A_\phi\succeq\frac{\binom{n-2k}{r-k}}{\binom{n-k}{r-k}}I.
\end{equation*}
Here a binomial coefficient with lower argument outside its range is zero, and $n$ is sufficiently large.
Indeed, at two $k$-sets with intersection size $t$, the first identity is Vandermonde's formula
\[
 \binom{n-2k+t}{r-2k+t}
 =\sum_{j=0}^k\binom tj\binom{n-2k}{r-2k+j}.
\]
All coefficients are nonnegative and $\mathsf W_k=I$. The final ratio is $1-O_{k,r}(n^{-1})$, so it is at least $1/2$ for large $n$. Also $d_n^{-1}\le Rn^{k-r}$ for a constant $R=R(k,r)$. Thus Theorem~\ref{thm:quantitative} applies to this seed.

\end{remark}

\subsection{The spectral gap after coordinate restriction}\label{sec:restriction}
For the remaining structural estimates, fix $G$ and $\phi$ satisfying
the seed assumptions. Seed balance gives $(A_\phi)_{ee}=1$ and
$A_\phi\ones=\binom{r}{k}\ones$. Counting completions of $e\cup f$
gives $(A_\phi)_{ef}\le R/n^{k-|e\cap f|}$ for $e\ne f$.
Indeed, if $|e\cup f|\le r$, there are at most
$n^{r-|e\cup f|}$ completions, each with weight at most $Rn^{k-r}$;
otherwise the entry is zero. Every off-diagonal entry is at most
$R/n$. Writing $A_\phi=I+\Gamma_\phi$ and using the off-diagonal
row sum $\binom{r}{k}-1$, we obtain
$\max_e\|\Gamma_\phi e_e\|_2^2\le R(\binom{r}{k}-1)/n$.

The off-diagonal comparison matrix after percolation has a factor
$p^{-1}$. We need a lower spectral bound for its random coordinate
restriction, not an operator-norm approximation by an entrywise
expectation.

\begin{lemma}\label{lem:restriction}
Let $\Gamma$ be a deterministic symmetric $N\times N$ matrix with zero diagonal and $I+\Gamma\succeq\kappa I$, where $0<\kappa\le1$. Set
$s=\max_i(\Gamma^2)_{ii}$.
Let $p\in(0,1]$, and let $\Pi$ be the diagonal projection formed from
independent Bernoulli($p$) variables. If $s>0$, then with probability at least
$1-N\exp(-\kappa^3p^2/(192s))$,
\begin{equation*}
 I_{\operatorname{im}\Pi}+p^{-1}(\Pi\Gamma \Pi)|_{\operatorname{im}\Pi}
 \succeq\frac\kappa2 I_{\operatorname{im}\Pi}.
\end{equation*}
For $s=0$ the conclusion is deterministic.
\end{lemma}

\begin{proof}
Let $B=(-\Gamma)_+$ be the negative spectral part: on an eigenvector
of $\Gamma$ with eigenvalue $\lambda$, $B$ acts by
$\max\{-\lambda,0\}$. Then $0\preceq B\preceq(1-\kappa)I$ and
$\Gamma\succeq-B$. Set $\eta=\kappa p/4$ and write $B=B_>+B_\le$,
where $B_>$ retains the eigenvalues of $B$ exceeding $\eta$ and
sets the others to zero. Thus $0\preceq B_\le\preceq\eta I$.
Set $L=s/\eta$; then $(B_>)_{ii}\le\eta^{-1}(\Gamma^2)_{ii}\le L$
for every $i$.
Writing $v_i=B_>^{1/2}e_i$ and $\Pi_{ii}=\xi_i$,
\[
 B_>^{1/2}(\Pi-pI)B_>^{1/2}
 =\sum_i(\xi_i-p)v_iv_i^{\T}.
\]
These summands are independent, mean-zero and symmetric, have norm at most $L$, and have variance matrix at most
$pL\sum_i v_iv_i^{\T}=pLB_>\preceq pLI$.
Theorem~\ref{thm:Tropp-self}, with deviation $t=\kappa p/4$, gives failure probability at most
\[
 N\exp\Bigl(-\frac{t^2}{2pL+(2/3)Lt}\Bigr)
 \le N\exp\Bigl(-\frac{\kappa^3p^2}{192s}\Bigr).
\]
On its complement,
\[
 \norm{\Pi B_>\Pi}
 =\norm{B_>^{1/2}\Pi B_>^{1/2}}
 \le p(1-\kappa)+\kappa p/4.
\]
The equality uses the equality of the nonzero eigenvalues of $XX^{\T}$ and $X^{\T}X$, with $X=\Pi B_>^{1/2}$. Also $\Pi B_\le \Pi\preceq\eta \Pi$. Consequently, on $\operatorname{im}\Pi$,
\[
 I+p^{-1}\Pi\Gamma \Pi
 \succeq I-p^{-1}\Pi B_>\Pi-p^{-1}\Pi B_\le \Pi
 \succeq\frac\kappa2 I.
\]
If $B_>=0$, the argument is immediate; if $s=0$, then $\Gamma=0$.
\end{proof}

Applying Lemma~\ref{lem:restriction} with this column-norm bound
shows that its conclusion fails with probability at most
$n^k\exp(-\Omega_{\kappa,R,r,k}(np^2))$.
For $(k,r)=(2,3)$, the extension condition is
$np^2\ge C_{\mathrm{ext}}\log n$. In every other case, $a=(r-k)/(\binom{r}{k}-1)<1/2$
and $p\ge n^{-a}$, so $np^2\ge n^{1-2a}$ is polynomially large.
This failure probability is then smaller than any prescribed
negative power of $n$.

\subsection{A uniform bound on the population inverse}\label{sec:inverse}
The population mean need not be constant in a general host. To control
it, we require a maximum-norm bound for the inverse of a boundedly
tilted seed Gram matrix. The next lemma obtains this from the overlap
geometry by induction over stars indexed by subsets of an edge.

\begin{lemma}[Subset-star inverse]\label{lem:star-inverse}
Let $\lambda,L,R_0>0$ be fixed, let $\cE\subseteq\binom{[n]}{k}$,
and let $A$ be a symmetric matrix on $\R^{\cE}$. Suppose that
$A\succeq\lambda I$, every absolute row sum is at most $L$, and
$|A_{ef}|\le R_0/n^{k-|e\cap f|}$ for $e\ne f$. Then
$\norm{A^{-1}}_{\infty\to\infty} =O_{k,\lambda,L,R_0}(1)$.
\end{lemma}
\begin{proof}
Solve $Ax=y$ and put $y_*=\norm y_\infty$. For $T\subseteq[n]$ with $|T|=t\le k$, write
$\cE_T=\{e\in\cE:T\subseteq e\}$.
We prove, by induction on $t$, that
\begin{equation*}
 \norm{x_{\cE_T}}_2\le C_t n^{(k-t)/2}y_*.
\end{equation*}
For $t=0$, use the spectral gap and $|\cE|\le n^k$.
For $t>0$, partition the complement of $\cE_T$ into
$\cF_U=\{f\in\cE:f\cap T=U\}$, indexed by proper subsets $U\subsetneq T$.
Set $s=|U|$.
For a fixed $f\in\cF_U$, an edge $e\in\cE_T$ with
$|e\cap f|=s+h$ has at most $C_kn^{k-t-h}$ choices.
Each corresponding matrix entry has absolute value at most
$R_0n^{-(k-s-h)}$. Summing over the bounded number of possible $h$ gives
$\max_{f\in\cF_U}\sum_{e\in\cE_T}|A_{ef}| \le C_kR_0n^{s-t}$.
The absolute row sums of the same block are at most $L$, so
$\norm{A_{\cE_T,\cF_U}}\le\sqrt{C_kLR_0}\,n^{(s-t)/2}$.
Because $\cF_U\subseteq\cE_U$, the induction hypothesis bounds the norm of $x_{\cF_U}$ by $C_s n^{(k-s)/2}y_*$.
The principal matrix on $\cE_T$ still has gap at least $\lambda$. Its block equation therefore gives
\[
 \norm{x_{\cE_T}}_2
 \le\lambda^{-1}\Biggl(1+\sqrt{C_kLR_0}
          \sum_{s=0}^{t-1}\binom ts C_s\Biggr)n^{(k-t)/2}y_*.
\]
This proves the induction. For $t=k$ and $T=e\in\cE$, the star contains only $e$, yielding the asserted maximum-norm bound.
\end{proof}

For $y\in\R^{E(G)}$, define the population map by
$g_\phi(y)_e=\sum_{Q\ni e}\phi(Q)e^{y(E(Q))}-1$.
Thus $g_\phi=\nabla\Phi_{G,1}$ and $g_\phi(0)=0$; its derivative
$Dg_\phi$, the \emph{population Hessian}, is the weighted Gram
matrix with weights $\phi(Q)e^{y(E(Q))}$.

\begin{corollary}\label{cor:population-inverse}
Let $G,\phi$ satisfy the seed hypotheses of
Theorem~\ref{thm:quantitative}, and fix $M>0$. If $e^{-\binom{r}{k}M}\le\vartheta_Q\le e^{\binom{r}{k}M}$
for every host clique, then the matrix
$A_{\phi,\vartheta}=\sum_Q\phi(Q)\vartheta_Q\ones_Q\ones_Q^{\T}$
satisfies $\|A_{\phi,\vartheta}^{-1}\|_{\infty\to\infty}
= O_{k,r,R,\kappa,M}(1)$. In particular, this applies to
$\int_0^1Dg_\phi(sy)\,ds$ whenever $\|y\|_\infty\le M$.
\end{corollary}
\begin{proof}
Termwise positive semidefinite comparison gives
$A_{\phi,\vartheta}\succeq\kappa e^{-\binom{r}{k}M}I$.
Its row sums are at most $\binom{r}{k}e^{\binom{r}{k}M}$, and
the seed entry bound holds with $R$ replaced by
$Re^{\binom{r}{k}M}$. Apply Lemma~\ref{lem:star-inverse}.
\end{proof}

\section{Concentration tools}\label{sec:external}
We collect three concentration tools. The first converts replacement
square sums into concentration, allowing a rare exceptional event,
and is used in both the independent-column and clique arguments.
The matrix inequalities supply the raw spectral estimates.
Finally, an elementary upper-tail bound for nonnegative Bernoulli
polynomials provides the counting estimates used in the next section.

\subsection{Concentration with rare exceptions}
\label{sec:rare-exceptions}\label{sec:bipentropy}
\begin{lemma}[Concentration with rare exceptions]\label{lem:T13}\label{lem:rare-exception}
Let $X$ be a measurable function of independent inputs
$Z_1,\ldots,Z_N$, with $|X|\le M$ for some $M>0$. Let $X^{(j)}$ be
obtained by independently resampling input $j$. Write $a_+=\max\{a,0\}$
and let $\E'_j$ mean expectation over the fresh input with the original
sample held fixed. Put
\[
 \mathcal V_+(X)=\sum_j\E'_j[(X-X^{(j)})_+^2\mid Z_1,\ldots,Z_N].
\] Suppose $\mathcal V_+(X)\le v$ outside an event of probability at most $\delta$, and $\mathcal V_+(X)\le v_{\max}$ everywhere, where $v,v_{\max}\ge0$ and $0\le\delta\le1$. Then, for every $\lambda\ge0$,
\begin{equation}\label{eq:rare-entropy}
 \log\E e^{\lambda(X-\E X)}
 \le v\lambda^2+\frac{v_{\max}\delta\lambda}{2M}(e^{2M\lambda}-1).
\end{equation}
Also, $\Var X\le v+v_{\max}\delta$.
\end{lemma}
\begin{proof}
For a nonnegative random variable $Z$, write
$\Ent(Z)=\E[Z\log Z]-\E Z\log\E Z$, and let $\Ent_j(Z)$
denote its entropy with all inputs except the $j$th held fixed.
Entropy tensorization~\cite[Theorem~4.22]{BLM} gives
$\Ent(Z)\le\sum_j\E\Ent_j(Z)$ for independent inputs.
We also use the resampling form of the Efron--Stein inequality,
$\Var X\le\frac12\sum_j\E(X-X^{(j)})^2$;
see~\cite[Theorem~3.1]{BLM}.

Apply entropy tensorization to $Z=e^{\lambda X}$. For one coordinate, let $Z,Z'$ be independent copies under its conditional law. Jensen's inequality and pair symmetrization give
\[
 \Ent_j(Z)\le\operatorname{Cov}_j(Z,\log Z)
 =\frac12\E_{j,j'}[(Z-Z')(\log Z-\log Z')],
\]
where the last expectation is over both conditional copies. Apply this with $Z=e^{\lambda X}$ and $Z'=e^{\lambda X'}$. Exchanging the copies and using $1-e^{-t}\le t$ for $t\ge0$ yields
\[
 \Ent(e^{\lambda X})\le\lambda^2\E[e^{\lambda X}\mathcal V_+(X)].
\]
Splitting the exceptional event and using $X\in[-M,M]$ gives
\[
 \frac{\E[e^{\lambda X}\mathcal V_+(X)]}{\E e^{\lambda X}}
 \le v+v_{\max}\delta e^{2M\lambda}.
\]
For $\psi(\lambda)=\log\E e^{\lambda(X-\E X)}$, this is
$\lambda\psi'(\lambda)-\psi(\lambda)
\le\lambda^2(v+v_{\max}\delta e^{2M\lambda})$.
Integrating $(\psi(\lambda)/\lambda)'$ from zero to $\lambda$ proves \eqref{eq:rare-entropy}.

By the Efron--Stein inequality recalled above, $\Var X\le\frac12\sum_j\E(X-X^{(j)})^2$. Exchangeability of $(X,X^{(j)})$ makes the squared expectations of the positive and negative parts of $X-X^{(j)}$ equal. Hence
$\Var X\le\E \mathcal V_+(X)\le v+v_{\max}\delta$,
which is the second estimate.
\end{proof}

The following consequence packages the tail estimates used in both
proofs. We take $h=s$ for the column model and $h=\log n$ for cliques.

\begin{corollary}\label{cor:poly-tail}
Fix constants $a,z,D,C_0>0$. There are $D_0,\epsilon_0>0$,
depending only on these constants, such that the following holds
for every $h\ge1$. Under the hypotheses of
Lemma~\ref{lem:rare-exception}, if $|X|\le1$,
$v_{\max}\le C_0e^{ah}$, $\delta\le e^{-D_0h}$ and
$vh\le\epsilon_0$, then $\PP(X-\E X\ge z)\le e^{-Dh}$.
\end{corollary}

\begin{proof}
Take $\lambda=Lh$ with $L=2(D+1)/z$ and
$\epsilon_0\le z/(2L)$. Then $v\lambda^2\le\lambda z/2$.
Choose $D_0$ large enough that
$(C_0Lh/2)e^{-(D_0-a-2L)h}\le1$ for every $h\ge1$.
The exceptional term in \eqref{eq:rare-entropy}, applied with $M=1$,
is then at most one. Chernoff's argument gives
$\PP(X-\E X\ge z)\le e^{-(D+1)h+1}\le e^{-Dh}$.
\end{proof}

\subsection{Matrix concentration}
For a rectangular real matrix $Z$, we also write $Z^*$ for its
transpose, as in the statements of the matrix concentration results.
We state first the lower-tail matrix Chernoff inequality, used to
establish the raw gap in the model. The form below follows from
\cite[Corollary~5.2 and Remark~5.3]{Tropp}.

\begin{theorem}[Matrix Chernoff]\label{thm:Tropp-Chernoff}
Let $X_1,\ldots,X_N$ be independent symmetric $d\times d$ random matrices with $0\preceq X_j\preceq LI$ almost surely. If $\sum_j\E X_j\succeq m_0I$ with $m_0>0$, then, for $0<\varepsilon<1$,
\[
 \PP\biggl(\lambda_{\min}\Big(\sum_jX_j\Big)
       \le(1-\varepsilon)m_0\biggr)
 \le d\exp\Bigl(-\frac{\varepsilon^2m_0}{2L}\Bigr).
\]
\end{theorem}

The following is the symmetric matrix Bernstein inequality in \cite[Theorem~1.4]{Tropp}.

\begin{theorem}[Symmetric matrix Bernstein]\label{thm:Tropp-self}
Let $X_1,\ldots,X_N$ be independent symmetric $d\times d$ random matrices with $\E X_i=0$ and $\lambda_{\max}(X_i)\le L$ almost surely.  Put
$\sigma^2=\left\|\sum_i\E X_i^2\right\|$.
Then for every $t\ge0$,
\[
 \PP\biggl(\lambda_{\max}\Bigl(\sum_iX_i\Bigr)\ge t\biggr)
 \le d\exp\Bigl(-\frac{t^2}{2\sigma^2+\frac23Lt}\Bigr).
\]
\end{theorem}

The rectangular form below is \cite[Theorem~1.6]{Tropp}.

\begin{theorem}[Rectangular matrix Bernstein]\label{thm:Tropp-rect}
Let $Z_1,\ldots,Z_N$ be independent random $d_1\times d_2$ matrices with $\E Z_i=0$ and $\|Z_i\|\le L$ almost surely.  Define
$\sigma^2=\max\{\|\sum_i\E Z_iZ_i^*\|,\|\sum_i\E Z_i^*Z_i\|\}$.
Then for every $t\ge0$,
\[
 \PP\biggl(\Bigl\|\sum_iZ_i\Bigr\|\ge t\biggr)
 \le(d_1+d_2)\exp\Bigl(-\frac{t^2}{2\sigma^2+\frac23Lt}\Bigr).
\]
Consequently, there is an absolute constant $C>0$ such that, for every
$q\ge\max\{1,\log(d_1+d_2)\}$,
\[
 \biggl(\E\Bigl\|\sum_iZ_i\Bigr\|^q\biggr)^{1/q}
 \le C\bigl(\sqrt{q\sigma^2}+qL\bigr).
\]
\end{theorem}

The moment estimate in Theorem~\ref{thm:Tropp-rect} follows from its tail bound by integration: for $W=\|\sum_iZ_i\|$, use $\E W^q=q\int_0^\infty t^{q-1}\PP(W\ge t)\,dt$, splitting at $t=\sigma^2/L$.  The subgaussian and exponential parts of the Bernstein tail give the two terms displayed above, and $q\ge\max\{1,\log(d_1+d_2)\}$ absorbs the dimensional prefactor.

\subsection{Upper tails for nonnegative polynomials}
For the counting estimates below, we need only an upper-tail bound
for nonnegative multilinear Bernoulli polynomials whose nonempty
expected derivatives are polynomially smaller than the target scale.
%This follows from the stronger inequalities of Schudy andSviridenko, but a direct moment argument suffices.

\begin{lemma}\label{lem:polynomial-upper-tail}
Fix an integer $d\ge1$ and constants $C_0,\delta>0$.
There are $c>0$, depending only on $d,C_0$, and $n_0$, depending
only on $d,C_0,\delta$, such that the following holds for every
$n\ge n_0$ and $N\ge1$. Let $Z=Z(X_1,\ldots,X_N)$ be a multilinear
polynomial of degree at most $d$, with nonnegative coefficients, in
independent Bernoulli variables with arbitrary parameters. Write
$\partial_T$ for the mixed derivative with respect to the variables
indexed by $T\subseteq[N]$. If $M>0$, $\E Z\le C_0M$, and
\[
 \E[\partial_T Z]\le C_0Mn^{-\delta}
 \qquad\text{whenever}\qquad 1\le|T|\le d,
\]
then
\[
 \PP\bigl(Z\ge(C_0+1)M\bigr)
 \le \exp\bigl(-c n^{\delta/d}\bigr).
\]
\end{lemma}

\begin{proof}
Write $Z=\sum_{S\subseteq[N],\,|S|\le d}a_S X_S$, where $a_S\ge0$
and $X_S=\prod_{i\in S}X_i$. Put $p_i=\E X_i$ and
$p_S=\prod_{i\in S}p_i$, with both empty products equal to one.
Set $M_0=\E Z$ and, for $1\le j\le d$, let
$M_j=\max_{|T|=j}\E[\partial_T Z]$, with an empty maximum interpreted
as zero. We first prove the moment bound
\[
 \E Z^h\le
 \Bigl(M_0+\sum_{j=1}^{d}(dh)^jM_j\Bigr)^h
 \qquad \text{for integers } h\ge 1.
\]

For $U\subseteq[N]$, let $X^U$ be obtained from $X$ by replacing the
coordinates in $U$ by one, and put $F(U)=\E Z(X^U)$.
Multilinearity gives
\[
 F(U)=
 \sum_{\substack{T\subseteq U\\|T|\le d}}
       \Bigl(\prod_{i\in T}(1-p_i)\Bigr)\E[\partial_T Z]
 \le M_0+\sum_{j=1}^{d}\binom{|U|}{j}M_j.
\]
The empty derivative is $Z$. To check the identity, expand
$1=p_i+(1-p_i)$ in each monomial for the indices in $S\cap U$;
the resulting terms are precisely those in the displayed sum.
This argument also permits parameters $p_i=0$ or $p_i=1$.

For integers $h\ge2$, expand $\E Z^h$ over ordered lists of monomial supports,
allowing repetitions. Fix the first $h-1$ supports and let
$U=S_1\cup\cdots\cup S_{h-1}$.
Independence and the identity $X_i^a=X_i$ for every positive integer
$a$ show that the sum over the last support is
\[
 \sum_{S_h}a_{S_h}p_{U\cup S_h}
 =p_U\sum_{S_h}a_{S_h}p_{S_h\setminus U}
 =p_UF(U).
\]
Since $|U|\le d(h-1)$, this last factor is bounded by $p_UB_h$,
where $B_h=M_0+\sum_{j=1}^{d}(dh)^jM_j$.
All coefficients are nonnegative, so summing over the first supports
gives $\E Z^h\le B_h\E Z^{h-1}$. Since $\E Z=M_0\le B_1$
and $B_h$ is nondecreasing in $h$, iteration proves the moment bound.
No division by $p_U$ is used.

Choose a constant $c_0>0$, depending only on $d,C_0$, so that
$C_0d^{d+1}c_0^d\le1/2$, and take
$h=\lfloor c_0n^{\delta/d}\rfloor$.
For sufficiently large $n$, we have $h\ge1$ and
\[
 \sum_{j=1}^{d}(dh)^jM_j
 \le C_0Mn^{-\delta}d(dh)^d
 \le M/2.
\]
The moment bound and $M_0\le C_0M$ now give
$\E Z^h\le[(C_0+1/2)M]^h$.
Markov's inequality~\cite[equation~(2.1)]{BLM} therefore yields
\[
 \PP\bigl(Z\ge(C_0+1)M\bigr)
 \le\Bigl(\frac{C_0+1/2}{C_0+1}\Bigr)^h
 \le\exp(-c n^{\delta/d}),
\]
where $c>0$ depends only on $d,C_0$; in the last step use
$h\ge(c_0/2)n^{\delta/d}$ for sufficiently large $n$.
This proves the lemma.
\end{proof}

\section{Random loads, overlaps, and the spectral gap}\label{sec:random}
The main output of this section is Lemma~\ref{lem:robust}. It gives
nearly balanced loads, a positive Hessian gap, and the overlap bounds
needed for local stability. It also shows which bounds survive a
fixed number of edge changes. All estimates concern the unscaled
weights; no scaling potential is assumed to exist.

We first prove the required unweighted extension bounds in the complete
ambient hypergraph. We then distinguish the individual-overlap and
squared-overlap scales. Proposition~\ref{prop:weighted-concentration}
uses conditional matrix exposure to obtain the weighted load and
Hessian estimates. Finally, Lemma~\ref{lem:robust} combines these
estimates and proves their stability under edge changes.

\subsection{Unweighted extension bounds}\label{sec:counts}
Couple $G_p$ to an ambient $\mathbb H\sim G^{(k)}(n,p)$ by writing
$G_p=G\cap\mathbb H$. Let $I_e$ be the independent indicators of all
$k$-sets; the unused nonhost indicators are harmless auxiliary
variables. For $S\subseteq[n]$ with $k\le |S|\le r$, put
\[
 \operatorname{ext}(S)=\sum_{\substack{Q\supseteq S\\|Q|=r}}
       \prod_{h\in E(Q)\setminus E(S)}I_h.
\]
Let $W=\max_{k+1\le|S|\le r}\operatorname{ext}(S)$ and
$D_*=\max_{|e|=k}\operatorname{ext}(e)$.
All requirements internal to the root set are omitted. In particular,
$\operatorname{ext}(S)=1$ for $|S|=r$, and $W=1$ for $r=k+1$.
Set $\operatorname{ext}(S)=0$ when $|S|>r$.
For distinct $k$-sets $e,j$, define
\[
 N_0(e,j)=\sum_{\substack{Q\supseteq e\cup j\\|Q|=r}}
       \prod_{h\in E(Q)\setminus\{e,j\}}I_h.
\]
Dropping the remaining requirements internal to $e\cup j$ gives
$N_0(e,j)\le\operatorname{ext}(e\cup j)\le W$ whenever this count is
nonzero. Neither terminal indicator is conditioned to be present.

For roots of size $s$, write
$\lambda_s=n^{r-s}p^{\binom{r}{k}-\binom{s}{k}}$.
The exact extension mean is
$\binom{n-s}{r-s}p^{\binom{r}{k}-\binom{s}{k}}$ and is comparable to
$\lambda_s$. Discrete convexity of $\log\lambda_s$ gives
$L_*:=\max_{k+1\le s\le r}\lambda_s=\max\{1,\lambda_{k+1}\}=\mu\gamma_{k,r}$.
For $r\ge k+2$, recall $a=(r-k)/(\binom{r}{k}-1)$. Set
$a_2=(r-k-1)/(\binom{r}{k}-k-1)$, $\epsilon_*=(a-a_2)/2$,
and $\alpha=\min\{\epsilon_*(\binom{r}{k}-1),1-ka\}$.
The strict convexity used in \eqref{eq:strict-density} gives
$0<a_2<a$, so these constants are positive. Set
$p_*=n^{-a+\epsilon_*}$. Below this dividing density, larger-root
completion counts are uniformly bounded. Above it, both $\mu$ and
$np^k$ are polynomially large, allowing moment estimates.

\begin{lemma}\label{lem:root-envelope}
For fixed $k\ge2$, $r\ge k+1$ and $D>0$, there is $C=C(k,r,D)$
such that, uniformly when $\mu\ge C_{\mathrm{ext}}\log n$ and
$C_{\mathrm{ext}}\ge2$, with probability at least $1-n^{-D}$,
\[
 W\le u_*:=C\mu\sqrt{\frac{\gamma_{k,r}(n,p)}{\log n}}.
\]
For $r\ge k+2$ and $p\le p_*$, the stronger bound $W\le C$ holds
with the same failure probability. For $r=k+1$, $W=1$
deterministically. The constants and lower bound on $n$ are
independent of $C_{\mathrm{ext}}$.
\end{lemma}

\begin{proof}
Since $\gamma_{k,r}\ge1/\mu$, we have
$\mu\sqrt{\gamma_{k,r}/\log n}\ge\sqrt{\mu/\log n}
\ge\sqrt{C_{\mathrm{ext}}}$.
This proves the assertion for $r=k+1$ and shows that the stronger
sparse-range bound below implies the general envelope. For
$r\ge k+2$, we use a greedy selection of completions when
$p\le p_*$ and a fixed-moment estimate when $p>p_*$.

\proofstep{Step 1: The range $p\le p_*$.}

Fix a root set $S$ with $k+1\le s=|S|<r$. A newly selected completion
that introduces $j$ vertices outside $S$ and the union of previous
completions requires
$e_j=\binom{r}{k}-\binom{r-j}{k}$
new hyperedges: these are all the hyperedges of that completion that
meet a newly introduced vertex. They are disjoint from the union of
earlier requirements. Ignoring additional requirements on old
vertices can only increase the probability.
Since $1\le j\le r-s\le r-k-1$, averaging the decreasing increments
of $e_j$ gives
$e_j/j\ge(\binom{r}{k}-k-1)/(r-k-1)=1/a_2$.
Hence $n^jp^{e_j}\le n^{-\nu j}$ throughout this range, where
$\nu=(a-\epsilon_*)/a_2-1>0$.

Choose a fixed integer $J$ with $\nu J>D+r+2$.
If $\operatorname{ext}(S)>\binom{J+r}{r-s}$, greedily select
completions, each introducing a new vertex, until their union has at
least $J$ nonroot vertices. Such a choice is possible before stopping:
a union with fewer than $J$ nonroot vertices contains at most
$\binom{J+r}{r-s}$ candidate completions. The stopped sequence has at
most $J$ terms and between $J$ and $J+r$ nonroot vertices.

There are only constantly many abstract overlap patterns for such a
sequence. If a pattern introduces $j_1,\ldots,j_h$ new vertices at
successive steps and $v=\sum_i j_i\ge J$, its total probability at a
fixed root is at most
\[
 C_{k,r,J} n^v p^{\sum_i e_{j_i}}
 \le C_{k,r,J}\prod_i n^{-\nu j_i}
 \le C_{k,r,J}n^{-\nu J}.
\]
Here the newly required hyperedges at the different steps are disjoint.
A union bound over patterns and at most $rn^r$ root sets proves the
asserted bounded envelope. Only the selected completions must
introduce new vertices; arbitrary distinct completions need not do so.

\proofstep{Step 2: Fixed moments and the range $p>p_*$.}

For every fixed positive integer $h$ and every root set of size $s$,
\begin{equation}\label{eq:fixed-root-moment}
 \E\operatorname{ext}(S)^h\le C_{k,r,h}(1+\lambda_s)^h.
\end{equation}
To see this, expand the ordinary moment over ordered lists of $h$
completions, allowing repetitions. At each step the old union has at
most $hr$ vertices. If the next completion introduces $j$ new
vertices, its old vertices have only constantly many choices, and its
new choices and requirements contribute at most $n^jp^{e_j}$.
This remains valid for $j=0$, when the factor is one and further
requirements among old vertices may be ignored.
The sequence $j\mapsto\log(n^jp^{e_j})$ is convex, because $e_j$ has
decreasing increments and $\log p\le0$. Its maximum on
$0\le j\le r-s$ is therefore at an endpoint, giving $1$ or
$\lambda_s$. Summing the finitely many possibilities at each step
proves \eqref{eq:fixed-root-moment}. This argument holds for every
$0<p\le1$.

In the present range, the definitions of $p_*$ and $\alpha$ give
$\mu,np^k\ge n^\alpha$. Take a fixed $h$ with
$h\alpha/4>D+r+2$ and apply Markov's inequality and a union bound to
\eqref{eq:fixed-root-moment}. With the required failure probability,
$W\le C n^{\alpha/4}(1+\lambda_{k+1})
 \le2C\mu n^{\alpha/4}\gamma_{k,r}$.
Since $\gamma_{k,r}\le n^{-\alpha}$ in this range,
\[
 \frac{n^{\alpha/4}\gamma_{k,r}}{\sqrt{\gamma_{k,r}/\log n}}
 =n^{\alpha/4}\sqrt{\gamma_{k,r}\log n}
 \le n^{-\alpha/4}\sqrt{\log n}=o(1).
\]
This proves the asserted envelope. All estimates are uniform for
each deterministic density in the stated range.
\end{proof}

We next prove the two upper bounds needed for the conditional exposure.
Near the logarithmic extension scale, a disjoint-family argument
handles the punctured clique degree. The count with one further
hyperedge omitted has polynomially smaller derivative scales and is
controlled by Lemma~\ref{lem:polynomial-upper-tail} throughout the
required range.

\begin{lemma}\label{lem:root-inputs}
For every fixed $D>0$, a sufficiently large $C_{\mathrm{ext}}=C_{\mathrm{ext}}(k,r,D)$ ensures
that, uniformly when $\mu\ge C_{\mathrm{ext}}\log n$, with probability at least
$1-n^{-D}$ we have $D_*\le C\mu$ and
$p\sum_{j\ne e}N_0(e,j)\le C\mu$ for every ambient $k$-set $e$.
The constant $C$ depends only on $k,r,D$.
\end{lemma}

\begin{proof} We first bound the punctured clique degrees, treating the sparse range separately, and then apply the polynomial upper-tail lemma to the count with one additional hyperedge omitted.

\proofstep{Step 1: Punctured clique degrees.}

For $r=k+1$, the count $\operatorname{ext}(e)$ is binomial with
parameters $n-k,p^k$. Different added vertices give monomials using
disjoint sets of hyperedges. The scalar Bernstein bound~\cite[equation~(2.10)]{BLM} and a union bound over
at most $n^k$ roots therefore suffice.

Suppose $r\ge k+2$ and first consider $p\le p_*$. Outside probability
$n^{-D-2}$, Lemma~\ref{lem:root-envelope} gives $W\le C_D$ for a
fixed constant $C_D$. Every nonroot vertex belongs to at most $C_D$
realized completions through a given root edge $e$, because the count
through that vertex is at most $\operatorname{ext}(e\cup\{v\})$.
Consequently, a family of $\operatorname{ext}(e)$ completions contains
at least $\operatorname{ext}(e)/((r-k)C_D)$ completions whose nonroot
vertex sets are pairwise disjoint.

For such a family, the required hyperedge sets are disjoint: the
completions intersect only in $e$, whose indicator is omitted.
The expected number of unordered realized families of size $h$ is at
most
\[
 \frac{\bigl(\binom{n-k}{r-k}p^{\binom{r}{k}-1}\bigr)^h}{h!}
 \le\Bigl(\frac{\mathrm e\mu}{h}\Bigr)^h.
\]
Take $h=\lceil4\mathrm e\mu\rceil$ and then a sufficiently large
constant $C'_D$. The event
$\{\operatorname{ext}(e)>C'_D\mu,\ W\le C_D\}$ implies the existence
of such a family, so it has probability at most $e^{-c\mu}$.
If the requested family size is impossible, that event is empty.
Union-bounding over roots and using $\mu\ge C_{\mathrm{ext}}\log n$ proves the
claim in this range. The family was counted unconditionally; no
independence conditional on $W\le C_D$ is asserted.

For $p>p_*$, every nonempty mixed derivative of
$\operatorname{ext}(e)$ either vanishes or fixes at least $k+1$
vertices together with $e$. Dropping all requirements internal to
these fixed vertices bounds its expectation by $CL_*$. Also
$\E\operatorname{ext}(e)\le C\mu$ and
$\mu/L_*=1/\gamma_{k,r}\ge n^\alpha$.
Apply Lemma~\ref{lem:polynomial-upper-tail} with target scale $M=\mu$
and separation exponent $\delta=\alpha$, and then take a union bound
over the at most $n^k$ roots. This gives $D_*\le C\mu$ with
superpolynomially small failure probability.

\proofstep{Step 2: The extra omitted hyperedge.}

Set $T_e=p\sum_{j\ne e}N_0(e,j)$. Its mean is
$(\binom{r}{k}-1)\binom{n-k}{r-k}p^{\binom{r}{k}-1}=\Theta_{k,r}(\mu)$.
Let $\mathcal T$ be a nonempty set of $t$ differentiated hyperedges,
and let $s$ be the size of their vertex union together with $e$.
Derivatives containing $I_e$ vanish. For a nonzero derivative,
$k+1\le s\le r$, and each surviving monomial has expectation
$p^{\binom{r}{k}-1-t}$, including the coefficient $p$ in $T_e$.
There are at most $Cn^{r-s}$ choices of its completing $r$-set and
marked omitted hyperedge. Thus
$\E[\partial_{\mathcal T}T_e] \le Cn^{r-s}p^{\binom{r}{k}-1-t}$.
For $s<r$, the derivative avoids $e$ and
$t\le\binom{s}{k}-1$, giving the upper bound $C\lambda_s$.
For $s=r$, it also avoids the marked omitted hyperedge, so
$t\le\binom{r}{k}-2$ and the bound is $Cp$, not $C$.
By \eqref{eq:strict-density} and $p\ge n^{-a}$,
\[
 \frac{\lambda_s}{\mu}=\frac1{\mu_s}
 \le n^{-[s-k-a(\binom{s}{k}-1)]}\quad \text{for } k+1\le s<r,
 \qquad \frac p\mu\le n^{-a}.
\]
Every nonempty expected derivative is therefore at most
$C\mu n^{-\delta}$ for some fixed $\delta>0$.
When $r=k+1$, the range $s<r$ is empty and the last bound suffices.
Lemma~\ref{lem:polynomial-upper-tail}, with target scale $M=\mu$,
and a union bound over the roots prove the second upper bound for
every density under consideration.
\end{proof}

\subsection{Individual and squared overlaps}\label{sec:overlap-scales}
Lemma~\ref{lem:root-envelope} controls the largest overlap for the
conditional exposure below. The balancing argument needs a direct
bound on the sum of squared overlaps at scale $\gamma_{k,r}$.

\begin{remark}\label{rem:max-obstruction}
Even the graph case $k=2$, $r=4$ shows why the largest overlap cannot
be bounded at the smaller scale. At $p=n^{-1/3}$, one has
$\mu\gamma_{2,4}=1$, whereas the extension count of a fixed
three-vertex root is $\operatorname{Bin}(n-3,1/n)$ and tends to a
Poisson variable of mean one. For every fixed $C$, the probability
that $W>C$ is therefore bounded below by a positive constant.
Thus the envelope in Lemma~\ref{lem:root-envelope} cannot be strengthened to $W=O(\mu\gamma_{k,r})$ with
polynomially small failure probability. 
\end{remark}

The next lemma estimates the square sum rather than its maximum times a first-moment sum.

\begin{lemma}[Squared-overlap envelope]\label{lem:root-square}
For every fixed $D>0$, a sufficiently large $C_{\mathrm{ext}}=C_{\mathrm{ext}}(k,r,D)$ ensures that,
uniformly when $\mu\ge C_{\mathrm{ext}}\log n$, with probability at least $1-n^{-D}$,
\begin{equation}\label{eq:root-square}
 \max_e\sum_{j\ne e}(I_j+p)N_0(e,j)^2
 \le C_{k,r,D}\mu^2\gamma_{k,r}(n,p).
\end{equation}
The maximum and sum are over all ambient $k$-sets.
\end{lemma}

The proof is given in Appendix~\ref{app:squared-overlap}. The estimate
includes rows indexed by absent terminal edges and supplies the
$O(\gamma_{k,r})$ energy bounds used in the local scaling argument.
It does not assert this smaller scale for every individual overlap.

\subsection{Concentration of weighted loads and the Hessian}\label{sec:weighted}
Extend the seed coefficients to all $r$-sets by putting $b_Q=n^{r-k}\phi(Q)$ on host cliques and $b_Q=0$ otherwise. Then $0\le b_Q\le R$ and $\sum_{Q\ni e}b_Q=n^{r-k}$ for every host edge $e$. For each such edge, retained or not, define its punctured weighted load by
\[
 D_e(p)=\frac1\mu\sum_{Q\supset e}b_Q
             \prod_{j\in E(Q)\setminus\{e\}}I_j.
\]

Let $\Pi$ denote the projection onto retained host edges and let $\Gamma_{\phi}=A_{\phi}-I$. On the retained-edge space put
\[
 \mathsf S_p=I+p^{-1}(\Pi\Gamma_{\phi}\Pi)|_{\operatorname{im}\Pi},
 \qquad
 \mathsf H_p^0=p^{1-\binom{r}{k}}\sum_{Q\in\cK_r(G_p)}\phi(Q)\ones_Q\ones_Q^{\T}.
\]
Seed balance gives $\E D_e(p)=1$, and the diagonal of $\mathsf H_p^0$ is $D_e(p)$ on retained edges.

\begin{proposition}[Weighted load and Hessian concentration]\label{prop:weighted-concentration}
Fix $k,r,R,D$ and an accuracy $\eta>0$. For sufficiently large
$C_{\mathrm{ext}}$ and $n$, uniformly when
$\mu\ge C_{\mathrm{ext}}\log n$, with probability at least $1-n^{-D}$,
\begin{align}
 \max_{e\in E(G)}|D_e(p)-1|&\le\eta,
       \label{eq:weighted-diagonal}\\
 \norm{\mathsf H_p^0-\mathsf S_p}&\le\eta.
       \label{eq:weighted-motif}
\end{align}
\end{proposition}

\begin{proof}
Coloring separates the clique indicators into independent blocks.
The rooted extension bounds control the coefficient norms in each
block, allowing conditional matrix concentration. We average these
bounds over colorings and then use the scalar version for punctured
loads; no simultaneous good event over all colorings is needed.

\proofstep{Step 1: Reconstruct the off-diagonal matrix from color classes.} 

Let $\chi:[n]\to[r]$ be chosen uniformly from all $r^n$ colorings,
independently of the hypergraph. An $r$-set is \emph{transversal} when
its vertices receive all $r$ colors. For each possible terminal
intersection size $h$, with $\max\{0,2k-r\}\le h\le k-1$, fix two
distinct $k$-subsets $L_1,L_2$ of $[r]$ whose intersection has size $h$.
Rows are terminal edges of color type $L_1$ and columns are terminal
edges of color type $L_2$. An entry is zero unless the actual vertices
in the shared colors agree. Form the rectangular block $A$ for this
type and take its symmetric dilation
$\bigl(\begin{smallmatrix}0&A\\A^{\T}&0\end{smallmatrix}\bigr)$,
embedded in the all-$k$-set space.

For a fixed ordered terminal pair $(e,f)$ with $|e\cap f|=h$ and a
fixed completing $r$-set $Q$, the probability of representation in
this dilation is
$c_h^{\rm col}=2h!(k-h)!^2(r-2k+h)!/r^r$.
The factor two exchanges the terminal orientations; the factorials
assign the shared vertices, the two exclusive parts, and the remaining
vertices to their prescribed colors. Averaging over $\chi$,
multiplying by $(c_h^{\rm col})^{-1}$, and summing over the possible
$h$ reconstructs the off-diagonal fluctuation matrix exactly, with its
terminal masks retained. This identity holds for each realization.

\proofstep{Step 2: Bound one exposed block.} 

It remains to control the exposure terms in this reconstruction. Fix one terminal type and one coloring. Rows and columns are the
$k$-sets with the prescribed color types just defined, with zero
entries for incompatible common vertices. Reveal the two terminal edge blocks.  Order the remaining $\ell=\binom{r}{k}-2$ color-$k$-set blocks as $g_1,\ldots,g_\ell$.  If $Q$ is transversal and has the required terminal edges, write $J_h(Q)$ for the Bernoulli indicator of the unique edge of $Q$ in block $g_h$.  Then
\begin{equation*}
 \prod_{h=1}^{\ell}J_h(Q)-p^\ell
 =\sum_{t=1}^{\ell}p^{\ell-t}
    \Bigl(\prod_{h<t}J_h(Q)\Bigr)(J_t(Q)-p).
\end{equation*}
This elementary telescoping identity is the source of conditional independence: in the $t$th summand the future blocks have been replaced by their mean factors, and the current scalar block $g_t$ is independent of the exposed coefficients.

Let $\cF_{\chi,t}$ be generated by the two terminal blocks and by $g_1,\ldots,g_{t-1}$.  The candidate contribution is $b_QI_eI_f(\prod_hJ_h(Q)-p^\ell)$, so the terminal masks remain throughout. For an actual $k$-set $j$ in the block $g_t$, let $A_j$ be the rectangular matrix whose $(e,f)$ entry is
\begin{equation*}
 (A_j)_{ef}
 =\sum_{\substack{Q\supset e,f,j\\Q\text{ transversal of the prescribed type}}}
 b_Q I_eI_f p^{\ell-t}
 \prod_{h<t}J_h(Q),
\end{equation*}
where in the sum $j$ is the unique edge of $Q$ in $g_t$ and future blocks are represented by their mean factors.  Then $A_j$ is nonnegative and $\cF_{\chi,t}$-measurable, and the $t$th exposure term is exactly
$T_{\chi,t}=\sum_{j\in g_t}(I_j-p)A_j$.
Conditionally on $\cF_{\chi,t}$, the variables $(I_j-p)_{j\in g_t}$ are independent and mean-zero.

For a fixed terminal row $e$ and a fixed current-block $k$-set $j$, define $X_{e,j}$ by replacing every future factor $p$ in the definition of $A_j$ by the corresponding future Bernoulli indicator and summing over the same transversal candidates.  The current indicator $I_j$ is still omitted.  Since future blocks are independent of $\cF_{\chi,t}$,
$(A_j\one)_e=\E[X_{e,j}\mid\cF_{\chi,t}]$.
The roots formed by $e\cup j$ contain at least $k+1$ vertices, and at most $r$ whenever a completion exists.  Dropping the color restrictions, dropping all requirements internal to those root vertices, and using $b_Q\le R$ yields the pointwise domination
$0\le X_{e,j}\le R \operatorname{ext}(e\cup j)\le RW$.
The same argument, with the column terminal $f$ in place of $e$, gives
\begin{equation}\label{eq:weighted-rowcol-coeff}
 \max_j\bigl\{\|A_j\one\|_\infty,\|A_j^{\T}\one\|_\infty\bigr\}
 \le R\E[W\mid\cF_{\chi,t}].
\end{equation}

There is also an aggregate identity.  Let
$X_e=\sum_{j\in g_t}I_jX_{e,j}$.
Every realized transversal clique through the row terminal $e$ uses exactly one edge from $g_t$, hence is counted exactly once in $X_e$.  Since the current block is independent of all $X_{e,j}$,
\begin{align*}
 p\sum_j(A_j\one)_e
 &=p\sum_j\E[X_{e,j}\mid\cF_{\chi,t}]\notag\\
 &=\E\Bigl[\sum_jI_jX_{e,j}\,\Bigm|\,\cF_{\chi,t}\Bigr]
 =\E[X_e\mid\cF_{\chi,t}].
\end{align*}
After dropping weights and colors, $X_e\le RD_*$, and therefore
\begin{equation}\label{eq:weighted-aggregate-upper}
 p\max\biggl\{\Bigl\|\sum_jA_j\one\Bigr\|_\infty,
              \Bigl\|\sum_jA_j^{\T}\one\Bigr\|_\infty\biggr\}
 \le R\E[D_*\mid\cF_{\chi,t}].
\end{equation}

\proofstep{Step 3: Truncate the coefficients and average.}

Fix a polynomial failure exponent $s$ much larger than any exponent needed later, and let $u_*$ be as in Lemma~\ref{lem:root-envelope}. The complete-host bounds give $W\le u_*$ and $D_*\le C_{k,r,s}\mu$ outside events of probability at most $n^{-s}$. Since $W\le n^r$ deterministically,
\[
 \E[W\mid\cF_{\chi,t}]
 \le u_*+\E[W\one_{\{W>u_*\}}\mid\cF_{\chi,t}],
 \qquad
 \E[W\one_{\{W>u_*\}}]\le n^{r-s}.
\]
Markov's inequality therefore gives
$\PP\left(\E[W\mid\cF_{\chi,t}]>u_*+1\right)\le n^{r-s}$.
The same argument applies to $D_*$ with threshold $C_{k,r,s}\mu$. After enlarging $u_*$ by a constant, define
\[
 \cB_{\chi,t}=
 \Bigl\{\E[W\mid\cF_{\chi,t}]\le2u_*,
       \ \E[D_*\mid\cF_{\chi,t}]\le C_{k,r,s}'\mu\Bigr\}.
\]
Then $\cB_{\chi,t}\in\cF_{\chi,t}$ and $\PP(\cB_{\chi,t}^c)\le C_{k,r,s} n^{r-s}$.

To estimate the truncated coefficients, we use Cauchy--Schwarz in the following form. For a nonnegative rectangular matrix $A$ with row sums at most $L_{\rm row}$ and column sums at most $L_{\rm col}$,
\begin{equation*}
 AA^{\T}\preceq L_{\rm col}\diag(A\one),\qquad
 A^{\T}A\preceq L_{\rm row}\diag(A^{\T}\one),\qquad
 \norm A\le\sqrt{L_{\rm row}L_{\rm col}}.
\end{equation*}
Indeed, $x^{\T}AA^{\T}x=\sum_j(\sum_iA_{ij}x_i)^2\le L_{\rm col}\sum_i x_i^2\sum_jA_{ij}$. Transposing gives the second inequality, and the first gives the norm bound.

Replace $A_j$ by $A_j^g=\one_{\cB_{\chi,t}}A_j$ and write $T^g_{\chi,t}=\sum_{j\in g_t}(I_j-p)A_j^g=\one_{\cB_{\chi,t}}T_{\chi,t}$.  On $\cB_{\chi,t}$, \eqref{eq:weighted-rowcol-coeff} and the preceding nonnegative-matrix estimate give
$\|A_j^g\|\le2Ru_*$.
Furthermore,
\[
 A_j^g(A_j^g)^{\T}
 \preceq \| (A_j^g)^{\T}\one\|_\infty\diag(A_j^g\one)
 \preceq2Ru_*\diag(A_j^g\one).
\]
Since $\E(I_j-p)^2=p(1-p)\le p$, summing and applying \eqref{eq:weighted-aggregate-upper} yields
\begin{equation*}
 \Bigl\|\sum_j\E\big[(I_j-p)^2 A_j^g(A_j^g)^{\T}\mid\cF_{\chi,t}\big]\Bigr\|
 \le C_{k,r,s}R^2u_*\mu.
\end{equation*}
The transposed variance has the same bound. Apply Theorem~\ref{thm:Tropp-rect} conditionally on $\cF_{\chi,t}$, with summand norm at most $2Ru_*$ and variance parameter at most $C_{k,r,s} R^2u_*\mu$. Both matrix dimensions are at most $n^k$, so its moment estimate applies for $q\ge C\log n$. Averaging the conditional moment bound over the exposed blocks gives
\begin{equation}\label{eq:weighted-T-moment}
 \|T^g_{\chi,t}\|_{L^q(\mathrm{op})}
 \le C_{k,r,R,s}\bigl(\sqrt{qu_*\mu}+qu_*\bigr).
\end{equation}

We now average the blocks over the auxiliary coloring. The discarded part $T^b=T-T^g$ is controlled in first moment, not in $L^q$ for $q\asymp\log n$.  Deterministically each absolute row and column sum of $T$ is at most $C_{k,r,R}n^{r}$, and hence $\|T\|\le C_{k,r,R}n^r$.  Therefore
$\E\|T^b_{\chi,t}\|\le C_{k,r,R}n^r\PP(\cB_{\chi,t}^c) \le C_{k,r,R,s}n^{2r-s}$.
The estimates are uniform over every fixed coloring and exposure level. There are only $\binom{r}{k}-2$ levels and constant reconstruction factors. Minkowski's inequality gives the same $L^q$ bound as \eqref{eq:weighted-T-moment}, up to a constant depending on $k,r$, for the averaged good part. The triangle inequality and Fubini's theorem give an expectation of $O_{k,r,R,s}(n^{2r-s})$ for the averaged bad part. No simultaneous good event over all colorings is required. Choose $s=s(k,r,D)$ first so large that Markov makes the latter smaller than $n^{-20}\mu$ with probability at least $1-n^{-D-10}$. Next take $q=C_{k,D}\log n$ and use Markov's inequality on \eqref{eq:weighted-T-moment}. Dividing the reconstructed matrix by $\mu$ gives
\[
 C_{k,r,R,D}\biggl(\sqrt{\frac{u_*\log n}{\mu}}+\frac{u_*\log n}{\mu}\biggr)
 \le C_{k,r,R,D}\Bigl((\gamma_{k,r}\log n)^{1/4}+(\gamma_{k,r}\log n)^{1/2}\Bigr).
\]
For a fixed pair of terminal edges, the conditional mean of their off-diagonal entry, conditioned only on those two edges being present, is exactly $p^{-1}(\Gamma_{\phi})_{ef}$: once the two terminal edges are conditioned present, the remaining $\binom{r}{k}-2$ clique edges survive with probability $p^{\binom{r}{k}-2}$, and the normalization is $p^{1-\binom{r}{k}}$.  This is a two-terminal conditional mean, not an expectation
conditional on the entire projection $\Pi$. Indeed, $\Pi$ records all
retained host edges, so conditioning on it determines the retained
cliques. The comparison matrix $\mathsf S_p$ must therefore be justified by
the colored exposure above, not by a matrix concentration argument
that treats $\mathsf S_p$ as $\E[\mathsf H_p^0\mid \Pi]$. This completes the off-diagonal estimate.

\proofstep{Step 4: Bound the punctured loads.} 

We now prove \eqref{eq:weighted-diagonal} by a scalar exposure, keeping its normalization explicit. Fix a host edge $e$ and omit its terminal indicator. Prescribe the distinct
colors $1,\ldots,k$ on its vertices, and color every other vertex
independently with $r$ colors. Keep completions whose remaining
$r-k$ vertices receive colors $k+1,\ldots,r$ exactly once. A fixed
$Q\supset e$ passes this filter with probability $(r-k)!/r^{r-k}$.

Expose the $\binom{r}{k}-1$ side-edge blocks in order and use the same telescoping identity. At a fixed level the conditionally mean-zero scalar sum is
$\sum_{j\in g_t}(I_j-p)c_j$, where $c_j\ge0$ is measurable in the earlier blocks. Fully realizing its future edges, while omitting $j$, gives a variable $X_{e,j}$ for which
\[
 c_j=\E[X_{e,j}\mid\cF_{\chi,t}]
       \le R\E[W\mid\cF_{\chi,t}],
 \qquad
 p\sum_jc_j\le R\E[D_*\mid\cF_{\chi,t}].
\]
The first inequality follows by dropping requirements internal to $e\cup j$; the second restores the current indicator and counts each actual punctured completion once. Thus the same event, determined by the previously exposed blocks, gives $0\le c_j^g\le2Ru_*$ and
\[
 \sum_j\E[(I_j-p)^2(c_j^g)^2\mid\cF_{\chi,t}]
 \le (\max_jc_j^g)\,p\sum_jc_j^g
 \le C_{k,r,s} R^2u_*\mu.
\]
Applying the moment estimate of Theorem~\ref{thm:Tropp-rect} to these $1\times1$ matrices gives the scalar counterpart of \eqref{eq:weighted-T-moment}. The discarded scalar part has the same deterministic polynomial bound and is handled by its averaged first moment, exactly as above.

Averaging over the colors with factor $r^{r-k}/(r-k)!$ reconstructs
\[
 \sum_{Q\supset e}b_Q
       \Bigl(\prod_{f\in E(Q)\setminus\{e\}}I_f-p^{\binom{r}{k}-1}\Bigr).
\]
The original weighted count has expectation
$p^{\binom{r}{k}-1}\sum_{Q\supset e}b_Q=\mu$, by seed balance. Division by $\mu$ therefore yields the deviation of $D_e(p)$ from one. Choose the polynomial failure exponent large enough for a union bound over the at most $n^k$ host edges. The scalar estimate has the same order as the off-diagonal estimate.
By \eqref{eq:gamma-small}, increasing $C_{\mathrm{ext}}$ and then
$n$ makes their combined bound at most $\eta$. This proves
\eqref{eq:weighted-diagonal} and \eqref{eq:weighted-motif}.
\end{proof}

\subsection{A raw event stable under edge changes}
For distinct host edges $e,j$, define the normalized two-hole weighted count
\[
 b(e,j)=p^{1-\binom{r}{k}}\sum_{Q\supset e,j}\phi(Q)
                    \prod_{f\in E(Q)\setminus\{e,j\}}I_f.
\]
Set $b(e,e)=0$. A \emph{flip} toggles an edge indicator.
Independent resampling flips $I_j$ with conditional probability
$\omega_j=p$ if $I_j=0$ and $\omega_j=1-p$ if $I_j=1$;
otherwise it leaves the hypergraph unchanged. For a subhypergraph $F\subseteq G$, write $D_e(F)$, $b_F(e,j)$, and $\mathsf H_F(0)$ for the same punctured load, two-hole count, and raw Hessian evaluated on $F$, with $p$ and the normalization kept fixed.

\begin{lemma}\label{lem:robust}
Fix $0<\theta<1/8$, $D>0$, and a nonnegative integer $\ell_0$.
For sufficiently large $C_{\mathrm{ext}}$ and $n$, if $\mu_r(n,p)\ge C_{\mathrm{ext}}\log n$, then with
probability at least $1-n^{-D}$ the following holds at the original sample
$G_p$. Every host edge has punctured load $D_e(p)=1\pm\theta$,
the raw Hessian satisfies $\mathsf H_p^0\succeq\kappa I/4$, and
\begin{equation}\label{eq:overlap-bounds}
 \begin{gathered}
 \max_{e\ne j}b(e,j)\le C\sqrt{\gamma_{k,r}},\qquad
 \max_e\sum_{j\ne e}\omega_jb(e,j)\le C,\\
 \max_e\sum_{j\ne e}\omega_jb(e,j)^2\le C\gamma_{k,r},\qquad
 \max_e\sum_{\substack{j\in E(G_p)\\j\ne e}}b(e,j)^2\le C\gamma_{k,r}.
 \end{gathered}
\end{equation}
All maxima are over host edges, present or absent, and the unrestricted
sums are over other host edges. Every hypergraph $F$ obtained from $G_p$ by
at most $\ell_0$ host-edge flips has punctured loads $1\pm2\theta$,
maximum two-hole count at most $C_{\ell_0}\sqrt{\gamma_{k,r}}$, raw Hessian at least
$\kappa I/8$, and
\begin{equation}\label{eq:robust-row-square}
 \max_{e\in E(G)}\mathcal L_e(F)\le C_{\ell_0}\gamma_{k,r},
 \qquad
 \mathcal L_e(F)=\sum_{\substack{a\in E(F)\\a\ne e}}b_F(e,a)^2.
\end{equation}
The two $\omega_j$-weighted family bounds are asserted only at the
original sample; the density and normalization remain fixed
throughout the flips.
\end{lemma}

\begin{proof}
We first establish the estimates at the original sample, then show
deterministically that the local count, load, and spectral bounds
survive the permitted edge changes.

\proofstep{Step 1: Establish the bounds at the original sample.}

Combine the probability estimate following Lemma~\ref{lem:restriction}
with Proposition~\ref{prop:weighted-concentration}, taking
$\eta=\min\{\theta,\kappa/4\}$. This gives the punctured-load bounds
and the gap $\kappa/4$.
Intersect with the events from Lemmas~\ref{lem:root-inputs},
\ref{lem:root-envelope}, and~\ref{lem:root-square}, choosing their failure
exponents large enough to absorb the finite union.

Since $b(e,j)\le R N_0(e,j)/\mu$, the individual bound follows from
Lemma~\ref{lem:root-envelope}. For the first weighted sum, the absent part is at
most $(R/\mu)p\sum_jN_0(e,j)=O(1)$, while the present part follows from
$\sum_{j\ne e}I_jb(e,j)=(\binom{r}{k}-1)D_e(p)$.
For the square sums use $\omega_j\le I_j+p$ and
Lemma~\ref{lem:root-square} directly:
\[
 \sum_{j\ne e}(I_j+p)b(e,j)^2
 \le\frac{R^2}{\mu^2}\sum_{j\ne e}(I_j+p)N_0(e,j)^2
 \le C\gamma_{k,r}.
\]
This proves both squared bounds, including the one over present edges.

Fix a realization in this event. In the remaining argument $p$, $\mu$,
and $\gamma_{k,r}$ stay fixed.

\proofstep{Step 2: Control individual counts and loads after flips.}

For an arbitrary ambient hypergraph $J$, write $\operatorname{ext}_J(S)$ for
its ignored-root extension count. Flipping a $k$-edge $j$ changes no requirements
if $j\subseteq S$. Otherwise each newly created completion contains
$S'=S\cup j$. Dropping all requirements internal to $S'$ gives
$\operatorname{ext}_{J'}(S) \le\operatorname{ext}_J(S)+\operatorname{ext}_J(S')$.
The second term is zero when $|S'|>r$. Hence $W(J')\le2W(J)$.
After at most $\ell_0$ flips the two-hole maximum is at most
$C_{\ell_0}\sqrt{\gamma_{k,r}}$. A punctured load ignores its own terminal indicator,
and a distinct flip satisfies
$|D_e(F')-D_e(F)|\le b_F(e,j)$.
Thus all loads remain $1\pm2\theta$ once $\sqrt{\gamma_{k,r}}$ is small enough.

\proofstep{Step 3: Control squared rows after flips.}

For distinct $e,a,j$ define the three-hole count
\[
 T_F(e,a,j)=p^{1-\binom{r}{k}}
   \sum_{Q\supset e,a,j}\phi(Q)
       \prod_{h\in E(Q)\setminus\{e,a,j\}}I_h(F).
\]
Flipping $j$ changes $b_F(e,a)$ in absolute value by $T_F(e,a,j)$.
Restoring the indicator of $a$ gives the exact counting identity
\[
 \sum_{\substack{a\in E(F)\\a\ne e,j}}T_F(e,a,j)
       =(\binom{r}{k}-2)b_F(e,j).
\]
For addition the common present coordinates are $E(F)$, and for deletion
they are $E(F)\setminus\{j\}$; thus the identity applies to precisely
the common coordinates in both cases. Since all intermediate two-hole
counts are $O_{\ell_0}(\sqrt{\gamma_{k,r}})$, the change in the common-coordinate sum of
squares is $O_{\ell_0}(\gamma_{k,r})$, by $|x^2-y^2|\le(x+y)|x-y|$.
The changed summation coordinate $a=j$ contributes at most another
$O_{\ell_0}(\gamma_{k,r})$. When $e=j$, the punctured two-hole counts and their
summation set are unchanged. Therefore
$|\mathcal L_e(F')-\mathcal L_e(F)|\le C_{\ell_0}\gamma_{k,r}$.
Boundedly many flips therefore preserve
\eqref{eq:robust-row-square}, after increasing $C_{\ell_0}$.

\proofstep{Step 4: Preserve the spectral gap.}

For a present edge $j$, let
$J_{Q,j}=\prod_{a\in E(Q)\setminus\{j\}}I_a(F)$, and let
$\ones_{Q\setminus j}$ indicate $E(Q)\setminus\{j\}$ on the common
coordinate space. Deleting $j$ removes the positive semidefinite matrix
\[
 \mathsf R_j=p^{1-\binom{r}{k}}\sum_{Q\ni j}\phi(Q)J_{Q,j}
          \ones_{Q\setminus j}\ones_{Q\setminus j}^{\T}.
\]
Its row sums are at most $(\binom{r}{k}-1)b_F(e,j)$, so
$\|\mathsf R_j\|\le C_{\ell_0}\sqrt{\gamma_{k,r}}$.
The new Hessian is a principal submatrix of the old one minus $\mathsf R_j$;
deletion costs at most $C_{\ell_0}\sqrt{\gamma_{k,r}}$ in its lower spectral bound.

For addition, order the new coordinate last. The enlarged Hessian is
\[
 \mathsf H_{F+j}(0)=
 \begin{pmatrix}
 \mathsf H_F(0)+\mathsf R_j&b_j\\
 b_j^{\T}&D_j(F)
 \end{pmatrix},\qquad \mathsf R_j\succeq0.
\]
Its border has $(b_j)_e=b_F(e,j)$ on old coordinates. Applying
\eqref{eq:robust-row-square} with terminal $j$ gives
$\|b_j\|_2^2\le C_{\ell_0}\gamma_{k,r}$. Consequently
\[
 \lambda_{\min}(\mathsf H_{F+j}(0))
 \ge\min\{\lambda_{\min}(\mathsf H_F(0)),D_j(F)\}
       -C_{\ell_0}\sqrt{\gamma_{k,r}}.
\]
The new diagonal is at least $1-2\theta>\kappa/4$.
Starting at gap $\kappa/4$, the total loss is at most
$C_{\ell_0}\sqrt{\gamma_{k,r}}<\kappa/8$ for sufficiently large $n$.
All these conclusions hold deterministically on the original event,
simultaneously for every permitted flip sequence. No family bound at
forced neighbors, and no union bound over those neighbors, has been used.
\end{proof}

\section{Local stability and the sum of squared changes}\label{sec:local}
This section proves two results from the raw estimates of
Lemma~\ref{lem:robust}. Proposition~\ref{prop:one-edge} constructs a
nearby exact potential after any one edge change. It bounds the
maximum coordinate change by $O(\sqrt\gamma)$ and the squared Euclidean
change by $O(\gamma)$. Proposition~\ref{prop:square} then fixes an
output coordinate and sums its squared changes over all possible
edge replacements. This sum is also $O(\gamma)$.

Lemma~\ref{lem:response-energy} isolates the deterministic argument
for the second result: a common Hessian transfers the forcing square
sums to the nonlinear responses. Throughout this section a bounded
potential is given; the existence argument is in
Section~\ref{sec:cap-setup}.

Fix a density $q$ and a hypergraph $F\subseteq G$. All normalizations in
this section use $q$, even when edges of $F$ are changed. Write
$\gamma=\gamma_{k,r}(n,q)$.

Fix a constant $C_{\mathrm{raw}}\ge1$ and a tolerance
$0<\theta<1/8$. We say that $F$ satisfies the \emph{local raw bounds}
if, for every $F'\subseteq G$ obtained from $F$ by at most one
host-edge flip,
\begin{equation}\label{eq:local-raw}
 \begin{gathered}
  \max_{e\in E(G)}|D_e(F')-1|\le2\theta,
  \qquad
  \mathsf H_{F'}(0)\succeq\frac{\kappa}{8}I,\\
  \max_{e\in E(G)}\mathcal L_e(F')
       \le C_{\mathrm{raw}}\gamma.
 \end{gathered}
\end{equation}
Here $\mathcal L_e(F')$ is the squared-row sum defined in
\eqref{eq:robust-row-square}, and the Hessian acts on
$\R^{E(F')}$. The density $q$, the normalization, and
$C_{\mathrm{raw}}$ remain fixed throughout this neighborhood.
These bounds concern the weights at zero potential and do not assume
that a stationary potential exists.
In particular, $b_F(e,j)\le\sqrt{C_{\mathrm{raw}}\gamma}$ for
$e\ne j$: use the term indexed by $j$ in $\mathcal L_e(F)$ when
$j$ is present, and in $\mathcal L_e(F+j)$ otherwise, since
$b_{F+j}(e,j)=b_F(e,j)$.

Proposition~\ref{prop:square} additionally requires the two
replacement-family bounds
\begin{equation}\label{eq:replacement-family}
 \max_{e\in E(G)}
 \sum_{\substack{j\in E(G)\\j\ne e}}\omega_jb_F(e,j)
 \le C_{\mathrm{raw}},
 \qquad
 \max_{e\in E(G)}
 \sum_{\substack{j\in E(G)\\j\ne e}}\omega_jb_F(e,j)^2
 \le C_{\mathrm{raw}}\gamma,
\end{equation}
where $\omega_j=q$ for $j\notin E(F)$ and $\omega_j=1-q$
for $j\in E(F)$. These two bounds are required only at $F$,
not at the modified hypergraphs.

Use the objective $\Phi_{F,q}$, the weights $w_Q(v)$ and the
extended potential $\widehat v$ from Section~\ref{sec:potential-setup}.
For an absent edge $e$, write $f_e(v)=\nabla\log s_e(v)$ whenever
the completion sum is positive. The punctured-load bound in \eqref{eq:local-raw} ensures this
positivity throughout the working box. We suppress $q$ in the
objective and write $\Phi_F=\Phi_{F,q}$.

Write $\mathsf D_F(v)=\diag(\mathsf H_F(v))$. If $\norm v_\infty\le r_0$, for a fixed small $r_0>0$, termwise comparison gives
\begin{equation}\label{eq:working-Hessian}
 cI\preceq \mathsf H_F(v)\preceq CI,\qquad
 \mathsf D_F(v)_{ee}\ge c,\qquad
 \max_e\norm{(\mathsf H_F(v)-\mathsf D_F(v))e_e}_2^2\le C\gamma.
\end{equation}
The last bound follows from the squared-row bound in
\eqref{eq:local-raw} and bounded tilting. The same estimates hold
at every one-edge neighbor. At an exact solution the diagonal is $\mathsf D_F(u)=I$, and
$\mathsf H_F(u)\one=\binom{r}{k}\one$. Since the matrix is nonnegative,
its operator norm is $\binom{r}{k}$.

\subsection{Stability under one edge change}
For a host edge $j$, write $F^{(j)}$ for the hypergraph obtained by
flipping its indicator. If $u$ and $u^{(j)}$ are exact potentials,
define their displacement on the original coordinate space by
\[
 \delta^{(j)}
   =(\widehat u^{(j)}-\widehat u)|_{E(F)}.
\]
When $j$ is deleted, its coordinate records the change from its
actual value to its virtual value. Keeping this coordinate will
let every response use the same original Hessian.

\Needspace{9\baselineskip}
\begin{proposition}[One-edge stability]\label{prop:one-edge}
Fix $C_{\mathrm{raw}}\ge1$ and a sufficiently small working radius
$r_0>0$. There are constants $a_0,C>0$, with $a_0<r_0$, such that
the following holds when $\theta,\gamma$ are sufficiently small.
Suppose that $F$ satisfies the local raw bounds
\eqref{eq:local-raw} and has an exact potential $u$ with
$\|u\|_\infty\le a_0$. Then $u$ is the unique exact potential of
$F$, and every $F^{(j)}$ has a unique exact potential $u^{(j)}$
with $\|u^{(j)}\|_\infty\le r_0$. Their displacement and extended
potentials satisfy
\[
 \|\delta^{(j)}\|_2^2\le C\gamma,
 \qquad
 \|\widehat u^{(j)}-\widehat u\|_\infty\le C\sqrt\gamma.
\]
\end{proposition}

\begin{proof}
We first bound the load error caused by one edge flip. Minimizing
the new objective in a small Euclidean ball gives a nearby exact
potential. We then compare its virtual coordinates to those of
the original solution.

\proofstep{Step 1: Bound the error caused by a flip.}

Put $m=\binom{r}{k}$, and let $F_-$ be the smaller of $F$ and
$F^{(j)}$. For a potential $v$ on their common edge set $E(F_-)$,
write
\[
 s_j(v)=\sum_{\substack{Q\in\cK_r(G),\ j\in E(Q)\\
              E(Q)\setminus\{j\}\subseteq E(F_-)}}
           \phi(Q)e^{v(E(Q)\setminus\{j\})},
 \qquad t_j(v)=\log\frac{q^{m-1}}{s_j(v)}.
\]
The actual coordinate of $j$ in any exact potential on $F_-+j$
equals $t_j(v)$; this is also its virtual coordinate on $F_-$.
Let $z_Q$ be the incidence vector of $E(Q)\setminus\{j\}$ on
the common edge set. Then
\[
 f_j(v):=\nabla\log s_j(v)
     =\sum_Q\pi_Q(v)z_Q,\qquad
 \pi_Q(v)=\frac{\phi(Q)e^{v(E(Q)\setminus\{j\})}}{s_j(v)}.
\]
The probabilities $\pi_Q$ give $\|f_j(v)\|_1=m-1$.
Throughout the working box, the denominator is comparable to
$q^{m-1}$ by the punctured-load bound. The numerator of coordinate
$a$ is bounded by the corresponding tilted two-hole count.
Since that count omits the indicator of $j$, it is the same for
$F$ and $F_-$. Consequently the local raw bounds give
\begin{equation}\label{eq:forcing-bounds}
 (f_j(v))_a\le Cb_F(a,j),\qquad
 \|f_j(v)\|_1=m-1,\qquad
 \|f_j(v)\|_2^2\le C\gamma.
\end{equation}
The last inequality uses the squared row with terminal $j$ and
common present edges. These estimates hold throughout the working
box, so they also apply along the segments used below.

Let $v$ now be the common-coordinate restriction of the original
potential $u$. Form a starting vector $\widetilde u$ for the new
objective $\Phi_{F^{(j)}}$: restrict $u$ when deleting $j$, and
append $t_j(v)$ when adding it. The latter value is at most
$(m-1)a_0+4\theta$ in absolute value. Taking
$a_0\le r_0/(8m)$ and $\theta\le r_0/32$ ensures
$\|\widetilde u\|_\infty\le r_0/4$.
These choices depend only on $m,r_0$; the subsequent smallness
threshold for $\gamma$ may also depend on $C_{\mathrm{raw}}$.

The new or removed clique weights through $j$ sum to one and have
side-edge load vector $f_j(v)$. Since
$\widetilde u=\widehat u|_{E(F^{(j)})}$, this gives the identity
\begin{equation}\label{eq:flip-gradient}
 \nabla\Phi_{F^{(j)}}
       \bigl(\widehat u|_{E(F^{(j)})}\bigr)
   =\varepsilon_j f_j(v),\qquad
 \varepsilon_j=
 \begin{cases}
  1,&j\notin E(F),\\
  -1,&j\in E(F),
 \end{cases}
\end{equation}
where the forcing is extended by zero at the added coordinate. In particular,
the new gradient has Euclidean norm $O(\sqrt\gamma)$.
We will use the same identity with the two hypergraphs exchanged.

\proofstep{Step 2: Construct the replacement.}

The local raw bounds give
$\nabla^2\Phi_{F^{(j)}}\succeq cI$ throughout the working box.
Minimize this objective on the Euclidean ball of radius
$L\sqrt\gamma$ about $\widetilde u$, where $L$ is a sufficiently
large fixed constant. For small enough $\gamma$, the ball lies
inside the working box. On its boundary,
\[
 \Phi_{F^{(j)}}(\widetilde u+h)-\Phi_{F^{(j)}}(\widetilde u)
 \ge c\|h\|_2^2-C\sqrt\gamma\|h\|_2>0.
\]
The minimizer $u^{(j)}$ is therefore interior and stationary, with
$\|u^{(j)}-\widetilde u\|_2\le L\sqrt\gamma$.
Let $v'$ be its restriction to the common edge set.
Since $\|\nabla t_j\|_1=m-1$, the mean-value bound gives
\[
 |t_j(v')-t_j(v)|\le(m-1)\|v'-v\|_\infty\le C\sqrt\gamma.
\]
Thus the displacement $\delta=\delta^{(j)}$ on the original
coordinate space has $\|\delta\|_2^2\le C\gamma$, including
the virtual coordinate when $j$ is deleted. This also bounds
$\|\delta\|_\infty$ and the flipped edge's own extended-coordinate
change by $C\sqrt\gamma$.

Positive tilts preserve the supported clique Gram-matrix kernel,
so the raw gap makes every full objective strictly convex at
every finite potential. Its exact potential is globally unique.

\proofstep{Step 3: Compare the remaining virtual coordinates.}

The extension in \eqref{eq:virtual-definition} satisfies, for
every host edge $e$,
\begin{equation}\label{eq:all-host-equation}
 q^{1-m}\sum_{Q\supset e}\phi(Q)J_{Q,e}
             e^{\widehat u(E(Q))}=1,\qquad
 J_{Q,e}=\prod_{j\in E(Q)\setminus\{e\}}I_j.
\end{equation}
For an originally absent edge $e$, put $f_e=f_e(u)$ at the original
potential. For a flip $j\ne e$, split its completion sums into
common and changed families. On common
completions, all side edges remain present. Bounded tilting and
the mean-value bound for the logarithm give a contribution at
most $C\sum_a(f_e)_a|\delta_a^{(j)}|$. Every changed completion
contains both $e$ and $j$, and its normalized contribution is at
most $Cb_F(e,j)$. Thus
\begin{equation}\label{eq:virtual-bound}
 |\widehat u_e^{(j)}-\widehat u_e|
 \le C\sum_{a\in E(F)}(f_e)_a|\delta_a^{(j)}|
                  +Cb_F(e,j).
\end{equation}
For deletion the summation includes the nonnegative term at $j$;
common completions do not use that edge, and removed completions
are already covered by the last term. Since $\|f_e\|_1=m-1$,
the bound is $O(\sqrt\gamma)$. Originally present coordinates are
covered by $\delta^{(j)}$, and the own flip was bounded in Step~2. Therefore
$\|\widehat u^{(j)}-\widehat u\|_\infty\le C\sqrt\gamma$.
\end{proof}

\subsection{The sum of squared replacement changes}
Proposition~\ref{prop:one-edge} constructs each replacement and bounds
its displacement energy. We now apply the deterministic response
estimate, Lemma~\ref{lem:response-energy}, with the original Hessian
$\mathsf H=\mathsf H_F(u)$ and Taylor envelope $\mathsf B=\mathsf H$.
For each host edge $j$, let $f_j$ be the completion forcing at the
original common-coordinate potential. For a present $j$, use the
completion function on $F-j$ and extend the vector by zero at $j$.
As in \eqref{eq:replacement-family}, the replacement weights are
$\omega_j=q$ for absent edges and $1-q$ for present edges.

\begin{proposition}[Replacement square sums]\label{prop:square}
Under the hypotheses of Proposition~\ref{prop:one-edge}, assume also
that the replacement-family bounds \eqref{eq:replacement-family}
hold at $F$. For sufficiently small $\gamma$, we have
\[
 \max_{a\in E(F)}\sum_j\omega_j(\delta_a^{(j)})^2\le C\gamma,
 \qquad
 \sup_{a\in E(G)}\sum_j\omega_j(\widehat u_a^{(j)}-\widehat u_a)^2
 \le C\gamma.
\]
\end{proposition}

\begin{proof}
We verify the hypotheses of Lemma~\ref{lem:response-energy} on the
original present edges, and then include the remaining host edges.

\proofstep{Step 1: Bound the Hessian and the forcing family.}

The Hessian bounds \eqref{eq:energy-Hessian} follow from
\eqref{eq:working-Hessian}. For $A_f=\sum_j\omega_jf_jf_j^{\T}$,
bounded tilting and \eqref{eq:replacement-family} give
\[
 \max_a\sum_j\omega_j(f_j)_a\le C,
 \qquad \max_a(A_f)_{aa}\le C\gamma.
\]
Since $\|f_j\|_1=\binom{r}{k}-1$, the row sums of the symmetric
nonnegative matrix $A_f$ are bounded by a constant. Thus
$\|A_f\|\le C$. These are the two forcing-moment bounds needed
in Lemma~\ref{lem:response-energy}.
Proposition~\ref{prop:one-edge} supplies the responses and the energy
bound $E:=\max_j\|\delta^{(j)}\|_2^2\le C\gamma$.

\proofstep{Step 2: Apply the common response estimate.}

Fix a flip $j$ and put $\delta=\delta^{(j)}$. Let $v,v'$ be the
restrictions of $u,u^{(j)}$ to their common edge set, as in the
proof of Proposition~\ref{prop:one-edge}, and let $\varepsilon_j$
be the sign in \eqref{eq:flip-gradient}. Since
$u+\delta=\widehat u^{(j)}|_{E(F)}$, applying
\eqref{eq:flip-gradient} with the two hypergraphs exchanged gives
$\nabla\Phi_F(u+\delta)=-\varepsilon_j f_j(v')$.
Taylor expansion in the original objective therefore gives the
same response formula for both flips:
\begin{equation}\label{eq:clique-local-response}
 \delta=\mathsf H^{-1}(g_j-\mathfrak r_j),\qquad
 g_j=-\varepsilon_j f_j(v'),\qquad |g_j|\le C f_j(v).
\end{equation}
Here $f_j(v')$ and $f_j(v)$ are extended by zero at a deleted $j$.
The forcing bound follows by bounded tilting on the common coordinates.
The Taylor remainder is taken in the \emph{original} objective:
\[
 (\mathfrak r_j)_a
 =\sum_{\substack{Q\in\cK_r(F)\\a\in E(Q)}}w_Q(u)
       \bigl(e^{\delta(E(Q))}-1-\delta(E(Q))\bigr).
\]
Since $0\le e^x-1-x\le Cx^2$ on the working range and
$\delta(E(Q))^2\le\binom{r}{k}\sum_{b\in E(Q)}\delta_b^2$,
\begin{equation}\label{eq:remainder-majorization}
 0\le\mathfrak r_j\le C\mathsf H\delta^{\circ2}.
\end{equation}

Equation~\eqref{eq:clique-local-response} uses the same original
Hessian $\mathsf H$ for every flip. Its forcing satisfies
$|g_j|\le C f_j$, and \eqref{eq:remainder-majorization} gives
$|r_j|\le C\mathsf H(\delta^{(j)})^{\circ2}$ with
$r_j=-\mathfrak r_j$. Lemma~\ref{lem:response-energy} therefore
applies with $\mathsf B=\mathsf H$ and, for sufficiently small
$\gamma$, gives
\[
 \mathcal V:=\max_{a\in E(F)}\sum_j\omega_j(\delta_a^{(j)})^2
 \le C\gamma.
\]

\proofstep{Step 3: Include the virtual coordinates.}

For an originally present output edge $e$, every replacement
is already covered by $\mathcal V$. For an originally absent $e$,
\eqref{eq:virtual-bound} and $\|f_e\|_1=\binom{r}{k}-1$ give
\[
 \begin{aligned}
 \sum_{j\ne e}\omega_j
   \Bigl(\sum_a(f_e)_a|\delta_a^{(j)}|\Bigr)^2
 &\le\Bigl(\binom{r}{k}-1\Bigr)
       \sum_a(f_e)_a\sum_j\omega_j(\delta_a^{(j)})^2\\
 &\le\Bigl(\binom{r}{k}-1\Bigr)^2\mathcal V.
 \end{aligned}
\]
The direct completion change has squared sum at most
$C\sum_{j\ne e}\omega_jb_F(e,j)^2\le C\gamma$ by
\eqref{eq:replacement-family}. For an originally absent $e$, its own addition contributes
$O(\gamma)$ by Proposition~\ref{prop:one-edge}.
This proves the second bound.
\end{proof}

\section{Concentration and the population mean}\label{sec:bootstrap}

We verify the cap improvement in Proposition~\ref{prop:bootstrap-inputs}.
Retain $G(q)$, $\tau_n$, $Z(q)$, $\bar Z(q)$ and $Y(q)$ from
Section~\ref{sec:cap-setup}. The proof has two parts: concentration
from Proposition~\ref{prop:square}, followed by the population mean
estimate.

\subsection{Concentration of the cap and cutoff coordinates}\label{sec:clique-concentration}
We apply the common rare-exception inequality from
Section~\ref{sec:rare-exceptions}, and in particular
Corollary~\ref{cor:poly-tail} with $h=\log n$ and $a=k$, on the
original product space.
No conditioning on the raw event is used.

\noindent\textit{Choice of constants.} For the fixed cap $\tau_n$,
$\gamma_{k,r}(n,q)\le\gamma_{k,r}(n,p)$ gives
$\frac{C\gamma_{k,r}(n,q)}{\tau_n^2}\log n\le C/C_{\mathrm{cap}}^2$.
First choose the working radius $r_0$ and a small raw-load tolerance
$\theta$. Choose a coordinate-tail exponent $D_Y>b+k+5$.
Next choose a cap-tail exponent $D_Z>b+k+5$ larger than both the
exceptional exponent required by Corollary~\ref{cor:poly-tail} for
the coordinate tails at deviation $1/32$ and $2(r-k)+k+10$.
Choose the raw exceptional exponent larger than the one required
for the cap tails at deviation $1/6$ and $D_Z+2$.

We use the raw event $\mathcal R(q)$ supplied by
Lemma~\ref{lem:robust}, with these parameters and robustness under
$\ell_0=\binom{r}{k}$ host-edge flips. Fix $C_{\mathrm{raw}}$
large enough to dominate all the local and family constants in that
lemma. This choice ensures that every hypergraph reached after at
most $\binom{r}{k}-1$ flips satisfies \eqref{eq:local-raw}: its
entire one-flip neighborhood lies within $\binom{r}{k}$ flips of
the original hypergraph. The replacement-family bounds
\eqref{eq:replacement-family} are used only at the original
hypergraph.

With these fixed raw bounds, let $a_0$ and the response constants be
as in Proposition~\ref{prop:one-edge}. Increase $C_{\mathrm{cap}}$
for all normalized-proxy bounds, and then increase
$C_{\mathrm{ext}}$ and $n_0$ until
$\tau_n\ll\min\{a_0,r_0/(8\binom{r}{k})\}$ and all required
smallness conditions hold. These choices depend only on
$k,r,R,\kappa,b$. Uniformly in $q\ge p$, monotonicity gives
\[
 \frac{\sqrt{\gamma_{k,r}(n,q)}}{\tau_n}
 \le\frac1{C_{\mathrm{cap}}\sqrt{\log n}}=o(1),
 \qquad \gamma_{k,r}(n,q)\ge n^{-(r-k)}.
\] All these parameters are fixed
for the remainder of this section.

Put $\cD(q)=\mathcal R(q)\cap\{Z(q)<\tau_n/2\}$.
The cap proxy will be controlled on $\mathcal R(q)$; the cutoff
coordinate proxies require the smaller event $\cD(q)$.

\begin{lemma}\label{lem:cap-concentration}
With the choices of constants at the start of
Section~\ref{sec:clique-concentration}, for every deterministic
$q\in[p,1]$ the bound $\mathcal V_+(Z(q))\le C\gamma_{k,r}(n,q)$ holds
on $\mathcal R(q)$. If $\bar Z(q)\le \tau_n/3$, then
$\PP(Z(q)\ge \tau_n/2)\le n^{-D_Z}$ and
$\Var Y_e(q)\le C\gamma_{k,r}(n,q)$ for every host edge $e$.
Moreover,
\[
 \max_{e\in E(G)}\PP\bigl(|Y_e(q)-\E Y_e(q)|>\tau_n/32\bigr)
 \le2n^{-D_Y}.
\]
\end{lemma}

\begin{proof}
Fix a deterministic $q$ and suppress it in $Z$, $Y$ and $\cD$.

\par\addvspace{3pt}\Needspace{2\baselineskip}%
\noindent\textit{Step 1: Bound the square proxy of the cap.}

Work on $\mathcal R(q)$ and use the signed-maximum comparison from
Step~1 of the model proof in Section~\ref{sec:model-concentration}.
If $Z<\tau_n$, a coordinate attaining the original signed maximum
bounds every positive cap decrement. Proposition~\ref{prop:square}
therefore gives $\mathcal V_+(Z)\le C\gamma_{k,r}(n,q)$.

If $Z=\tau_n$ and no flip lowers it, the proxy is zero. Otherwise a lowering neighbor has a small exact potential. Apply Proposition~\ref{prop:one-edge} from that neighbor back to the original hypergraph; Proposition~\ref{prop:one-edge} bounds its extended norm by $\tau_n+C\sqrt{\gamma_{k,r}(n,q)}$. Uniqueness identifies it with the one in the cap definition. Its maximizing signed coordinate is at least $\tau_n$ and again bounds every positive cap decrement. The original raw event supplies \eqref{eq:local-raw} and
\eqref{eq:replacement-family} at the original hypergraph, so
Proposition~\ref{prop:square} gives the same proxy estimate.
This reverse application uses only the local bounds \eqref{eq:local-raw}
at the lowering neighbor; it does not require the replacement-family
bounds there. These local bounds follow from the flip-count argument
at the start of this subsection. Globally, $\mathcal V_+(Z)\le \tau_n^2n^k$. This part does not assume $\bar Z(q)\le \tau_n/3$.

\proofstep{Step 2: Concentrate the cap and cutoff coordinates.} 

Now assume $\bar Z(q)\le \tau_n/3$. Apply Corollary~\ref{cor:poly-tail} to $Z/\tau_n$ at deviation $1/6$. The normalized good-event proxy is at most $C\gamma_{k,r}(n,q)/\tau_n^2$, and its product with $\log n$ is at most $C/C_{\mathrm{cap}}^2$. The choices above therefore give $\PP(Z\ge \tau_n/2)\le n^{-D_Z}$.

On $\cD$, every hypergraph reached by at most
$\binom{r}{k}-1$ forced flips satisfies \eqref{eq:local-raw}.
Starting from the original potential, whose extended norm is less
than $\tau_n/2$, repeated applications of
Proposition~\ref{prop:one-edge} change that norm by at most
$C\sqrt{\gamma_{k,r}(n,q)}=o(\tau_n)$ in total. All these potentials
therefore remain below the common cap and cutoff $\tau_n$ and,
by its choice, inside the initial subbox required at each step.

In particular, $Y$ and every one-input replacement agree with the corresponding extended exact potentials there. Proposition~\ref{prop:square} gives both signed proxies of each $Y_e$ at most $C\gamma_{k,r}(n,q)$ on $\cD$. Globally they are at most $4\tau_n^2n^k$.

The complement of $\cD$ has probability at most $n^{-D_Z}+\PP(\mathcal R(q)^c)$. By the exponent hierarchy, this is small enough for both applications of Lemma~\ref{lem:rare-exception}. Its variance bound gives
$\Var Y_e\le C\gamma_{k,r}(n,q)+4\tau_n^2n^k\PP(\cD^c) \le C\gamma_{k,r}(n,q)$,
where $\gamma_{k,r}(n,q)\ge n^{-(r-k)}$ absorbs the polynomial error. Applying Corollary~\ref{cor:poly-tail} to $\pm Y_e/\tau_n$ at deviation $1/32$ gives the claimed coordinate-tail bound. All comparisons hold with constants independent of the chosen deterministic $q$.
\end{proof}

\subsection{The mean estimate and completion of the proof}\label{sec:clique-bootstrap}

The remaining task is to locate the mean of the cutoff potential.
Resampling a clique's side edges gives a small population error;
the population inverse then bounds the mean coordinatewise.

\begin{proposition}\label{prop:mean-improvement}
With the choices of constants at the start of
Section~\ref{sec:clique-concentration}, let $q\in[p,1]$ be deterministic. If $\bar Z(q)\le \tau_n/3$, then
$\|\E Y(q)\|_\infty\le C\sqrt{\gamma_{k,r}(n,q)}$ and $\bar Z(q)\le \tau_n/8$.
\end{proposition}

\begin{proof}
Fix $q$ with $\bar Z(q)\le\tau_n/3$. Write $Y=Y(q)$, $Z=Z(q)$,
$\cD=\cD(q)$, $\overline y=\E Y$, and abbreviate
$\gamma=\gamma_{k,r}(n,q)$.
Lemma~\ref{lem:cap-concentration} gives
$\PP(\cD^c)\le n^{-D_Z}+\PP(\mathcal R(q)^c)$.
The exponent choices and $\gamma\ge n^{-(r-k)}$ therefore imply
\begin{equation}\label{eq:mean-exception}
 n^{r-k}\PP(\cD^c)=o(\gamma).
\end{equation}

\proofstep{Step 1: Decouple the clique indicators.}

For a host edge $e$ and a deterministic vector $y$, put
\[
 \Lambda_e(y)=q^{1-\binom{r}{k}}
       \sum_{Q\ni e}\phi(Q)J_{Q,e}e^{y(E(Q))}-1.
\]
Then $\E\Lambda(y)=g_\phi(y)$. We use the model decomposition
\eqref{eq:mean-error-decomposition} with
$g_{\mathrm{emp}}=\Lambda$ and $g_{\mathrm{pop}}=g_\phi$.

The cutoff term is supported on $\{Z=\tau_n\}$: on its complement,
\eqref{eq:all-host-equation} gives $\Lambda_e(Y)=0$, while on the
cutoff event $Y=0$ and
$|\Lambda_e(0)|\le q^{1-\binom{r}{k}}+1\le2n^{r-k}$.
Since $\{Z=\tau_n\}\subseteq\cD^c$,
\eqref{eq:mean-exception} yields $|\E\Lambda_e(Y)|=O(\gamma)$.

For each $Q\ni e$, resample all inputs in $E(Q)\setminus\{e\}$
independently, and recompute the cutoff vector $Y^{[Q,e]}$.
It has the law of $Y$ and is independent of the original product
$J_{Q,e}$, because it uses fresh side-edge inputs and unchanged
inputs outside that set. Thus, unconditionally,
\[
 \E[J_{Q,e}e^{Y^{[Q,e]}(E(Q))}]
       =q^{\binom{r}{k}-1}\E e^{Y(E(Q))}.
\]
On $\cD$, the successive replacements remain below the cutoff,
as proved in Lemma~\ref{lem:cap-concentration}, and give
$\|Y-Y^{[Q,e]}\|_\infty\le C\sqrt\gamma$ for every outcome of the fresh
inputs. The exponential is Lipschitz on the bounded range of clique
sums. Since $\E[J_{Q,e}\one_{\cD}]\le q^{\binom{r}{k}-1}$, we obtain
\[
 \left|\E[J_{Q,e}e^{Y(E(Q))}]
       -q^{\binom{r}{k}-1}\E e^{Y(E(Q))}\right|
 \le C\sqrt\gamma q^{\binom{r}{k}-1}+C\PP(\cD^c).
\]
Multiply by $\phi(Q)q^{1-\binom{r}{k}}$ and sum over $Q\ni e$.
Seed balance, $q^{1-\binom{r}{k}}\le n^{r-k}$, and
\eqref{eq:mean-exception} give
\[
 \|\E g_\phi(Y)\|_\infty\le C\sqrt\gamma+C\gamma\le C\sqrt\gamma.
\]

\proofstep{Step 2: Pass to the mean and invert.}

The coordinate variance bound in Lemma~\ref{lem:cap-concentration}
implies, for each clique $Q$,
\[
 \Var\bigl(Y(E(Q))\bigr)
 \le\binom{r}{k}\sum_{f\in E(Q)}\Var Y_f\le C\gamma.
\]
Taylor expansion of the exponential about $\overline y(E(Q))$
has zero expected linear term and a remainder bounded by this
variance. Consequently
$\E e^{Y(E(Q))}=e^{\overline y(E(Q))}+O(\gamma)$.
Summing with the seed weights through each edge gives
$\|g_\phi(\overline y)\|_\infty\le C\sqrt\gamma$. Since $g_\phi(0)=0$,
\[
 g_\phi(\overline y)
   =\left(\int_0^1Dg_\phi(s\overline y)\,ds\right)\overline y.
\]
The mean satisfies $\|\overline y\|_\infty\le\tau_n$, so
Corollary~\ref{cor:population-inverse} applies to this averaged
Hessian and yields $\|\overline y\|_\infty\le C\sqrt\gamma$.

\proofstep{Step 3: Improve the expected cap.}

A union bound in Lemma~\ref{lem:cap-concentration} and
$\|Y-\overline y\|_\infty\le2\tau_n$ give
$\E\|Y-\overline y\|_\infty
 \le\tau_n/32+O(\tau_n n^{k-D_Y})$.
Taking expectations in the common identity
\eqref{eq:cap-cutoff-identity} now gives
\[
 \begin{aligned}
 \bar Z(q)&\le\|\overline y\|_\infty
    +\E\|Y-\overline y\|_\infty+\tau_n\PP(Z=\tau_n)\\
 &\le C\sqrt\gamma+\tau_n/32+O(\tau_n n^{k-D_Y})
       +\tau_n n^{-D_Z}\le\tau_n/8.
 \end{aligned}
\]
The last inequality holds uniformly in $q$ for sufficiently large
$n$, since $\sqrt\gamma=o(\tau_n)$ and the tail exponents exceed $k$.
This proves the second assertion.\qedhere

\end{proof}

\begin{proof}[Proof of Proposition~\ref{prop:bootstrap-inputs}]
Choose the constants in the order specified at the start of
Section~\ref{sec:clique-concentration}. Under $\bar Z(q)\le \tau_n/3$, Lemma~\ref{lem:cap-concentration}
gives $\PP(Z(q)\ge \tau_n/2)\le n^{-D_Z}\le n^{-b}$.
Proposition~\ref{prop:mean-improvement} gives
$\bar Z(q)\le\tau_n/8$. These estimates are uniform for each deterministic
$q\in[p,1]$, as required.
\end{proof}

\medskip
\noindent\textbf{AI assistance.}
The author conceived the key ideas and used ChatGPT to assist with developing proofs and refining the exposition. The author takes full responsibility for the content of this paper.

\appendix
\section{The squared-overlap estimate}\label{app:squared-overlap}
This appendix proves Lemma~\ref{lem:root-square}. All counts are taken
in the complete ambient product space of Section~\ref{sec:counts}.
The near-threshold range uses the bounded larger-root envelope.
In the remaining range we apply Lemma~\ref{lem:polynomial-upper-tail}
to a polynomial counting pairs of overlapping cliques. Fully overlapping pairs are
included throughout.

\begin{proof}[Proof of Lemma~\ref{lem:root-square}]
Write $\gamma=\gamma_{k,r}(n,p)$ and let
$L=L_* =\mu\gamma$ and $M=\mu L=\mu^2\gamma$.
For $r=k+1$, use $N_0(e,j)\le1$, Lemma~\ref{lem:root-inputs}, and the
identity
$\sum_{j\ne e}I_jN_0(e,j) =\left(\binom{r}{k}-1\right)\operatorname{ext}(e)$
to obtain \eqref{eq:root-square}, since $M=\mu$.
For $r\ge k+2$ and $p\le p_*$, Lemma~\ref{lem:root-envelope} gives
$W=O(1)$ with any prescribed polynomial failure probability. The same
identity and Lemma~\ref{lem:root-inputs} then give
\[
 \sum_{j\ne e}(I_j+p)N_0(e,j)^2
 \le W\biggl[\Bigl(\binom{r}{k}-1\Bigr)\operatorname{ext}(e)
                   +p\sum_{j\ne e}N_0(e,j)\biggr]
 \le C\mu\le CM.
\]

It remains to consider $r\ge k+2$ and $p>p_*$. By the definitions in
Section~\ref{sec:counts}, $\mu,np^k\ge n^\alpha$, and hence
$M/L^2=1/\gamma\ge n^\alpha$.

Fix a root edge $e$ and denote the left-hand sum in
\eqref{eq:root-square} by $Z_e$. Replace every repeated indicator power
by the indicator itself, so $Z_e$ is a multilinear polynomial with
nonnegative coefficients, of degree at most $d=2\binom{r}{k}-k-2$.
Expand each square over ordered pairs $Q_1,Q_2$ of $r$-sets containing
$e,j$. Their intersection has size $s$ with $k+1\le s\le r$.
Let $\mathcal U=E(Q_1)\cup E(Q_2)$.
This marked pair contributes the two monomials
$\prod_{a\in\mathcal U\setminus\{e\}}I_a$ and
$p\prod_{a\in\mathcal U\setminus\{e,j\}}I_a$.
Both have expectation $p^{2\binom{r}{k}-\binom{s}{k}-1}$.
There are $O_{k,r}(n^{2r-s-k})$ such pairs at $e$ and only $O_{k,r}(1)$ choices
of their marked shared edge $j$. Consequently
\begin{equation*}
 \E Z_e\le C_{k,r}\sum_{s=k+1}^r
  n^{2r-s-k}p^{2\binom{r}{k}-\binom{s}{k}-1}
 =C_{k,r}\mu^2\sum_{s=k+1}^r\frac1{\mu_s(n,p)}\le C_{k,r}M,
\end{equation*}
with $\mu_s$ as defined in Section~\ref{sec:quantitative}.

We claim that $\E[\partial_T Z_e]\le C_{k,r}L^2$ for every
nonempty set $T$ of at most $d$ differentiated variables.
Here $T$ is a fixed set of ambient edge variables. Terms not containing
all variables in $T$ disappear on differentiation. For the others,
first specify the abstract overlap pattern, the marked edge $j$, and
the placement of the vertices of $T$ into that pattern. There are only
$O_{k,r}(1)$ possibilities: all patterns have at most $2r-k-1$ vertices and
$T$ has bounded size. The ambient locations of the vertices of $T$
are fixed by the derivative; one sums only over the remaining free
vertices, not over new choices of the differentiated edges. Inconsistent placements contribute zero.

The polynomial $Z_e$ does not depend on $I_e$, so derivatives with
$e\in T$ vanish. Otherwise choose an edge of $T$ and order the two
cliques so that the first contains it. The vertices of $e$ together with
the fixed derivative vertices give at least $k+1$ roots in the first
clique. Include all fixed derivative vertices belonging to this clique
in its root set, and call its size $t_1$. After this clique is placed,
root the second at all $s$ intersection vertices and all its fixed
derivative vertices; call this size $t_2$. Thus $k+1\le t_1,t_2\le r$.
Ignore all edge requirements internal to each of these root sets.
The remaining required edges from the two cliques are disjoint: those
of the second clique have a vertex outside the intersection.
Moreover every differentiated edge is internal to its corresponding
root set, so none of the remaining requirements is a differentiated
variable. The sum of the resulting probabilities over placements is
at most
\[
 C_{k,r}\bigl(n^{r-t_1}p^{\binom{r}{k}-\binom{t_1}{k}}\bigr)
       \bigl(n^{r-t_2}p^{\binom{r}{k}-\binom{t_2}{k}}\bigr)
 \le C_{k,r}L^2.
\]
For the second type of summand the edge $j$ is omitted. If it remains
among the displayed requirements, deleting it increases this upper
bound by at most $p^{-1}$, which is canceled by the coefficient $p$.
If it has already been dropped, that coefficient can simply be bounded
by one. This proves the derivative bound, including derivatives
of every order and completely overlapping clique pairs.

Apply Lemma~\ref{lem:polynomial-upper-tail} to $Z_e$ with target
scale $M$ and separation exponent $\delta=\alpha$.
The mean is at most $C_{k,r}M$, and every nonempty expected derivative
is at most $C_{k,r}L^2\le C_{k,r}Mn^{-\alpha}$.
The lemma therefore gives
$\PP(Z_e\ge C'_{k,r}M)\le\exp(-c_1n^{\alpha/d})$,
where $c_1>0$ depends only on $k,r$.
This is smaller than $n^{-D-k-2}$ for all sufficiently large $n$.
A union bound over the at most $n^k$ root edges completes the
dense-range proof. All constants and thresholds are uniform in the
deterministic density parameter.
\end{proof}

\end{document}